\documentclass{amsart}
\usepackage{amssymb,amsmath, amsthm, amsfonts}
\usepackage{mathrsfs,color}
\usepackage[utf8]{inputenc}
\usepackage{float}
\usepackage[all]{xypic}
\newcommand{\luk}{\L u\-ka\-si\-e\-w\-icz}
\usepackage{pgf,tikz}
\newtheorem{theorem}{Theorem}[section]
\newtheorem{lemma}[theorem]{Lemma}
\newtheorem{corollary}[theorem]{Corollary}
\newtheorem{proposition}[theorem]{Proposition}
\newtheorem{claim}[theorem]{Claim}
\usepackage{amsaddr}

\theoremstyle{definition}
\newtheorem{notation}[theorem]{Notation}
\newtheorem{definition}[theorem]{Definition}
\newtheorem{remark}[theorem]{Remark}
\newtheorem{example}[theorem]{Example}

\newcommand{\FPL}{{\rm FP(\L,\L)}}
\newcommand{\pfm}{{\bf PFm}}
\newcommand{\mo}{{\mathcal{M}od}}
\newcommand{\lu}{{\textrm \L}}
\newcommand{\free}{{\rm Free}}
\newcommand{\term}{{\rm Term}}
\newcommand{\bv}{{\bf v}}
\newcommand{\by}{{\bf y}}

\newcommand{\CMV}{\mathsf{coMV}}
\newcommand{\alg}[1]{{\textbf{\upshape #1}}}  %
\newcommand{\vv}[1]{\mathsf {#1}}
\newcommand{\cc}[1]{\mathcal {#1}}
\newcommand{\scr}[1]{\mathscr {#1}}

\newcommand{\oneset}{{\rm O}}
\newcommand{\rest}{{\upharpoonright_\mathscr{P}}}

\usepackage[left=3cm,right=3cm,top=2.5cm,bottom=2.5cm]{geometry}

\begin{document}
\title[Encoding  de Finetti's coherence within \L ukasiewicz logic and MV-algebras]{Encoding  de Finetti's coherence\\ within \L ukasiewicz logic and MV-algebras}
\author[T. Flaminio and S. Ugolini]{Tommaso Flaminio and Sara Ugolini}
\address{Artificial Intelligence Research Institute (IIIA), CSIC, Barcelona, Spain}
\email{\texttt{\{tommaso, sara\}@iiia.csic.es}}
\date{}

\maketitle
\begin{abstract}
The present paper investigates proof-theoretical and algebraic properties for the probability logic $\FPL$, meant for reasoning on the uncertainty of \luk\ events. Methodologically speaking, we will consider a translation function between formulas of $\FPL$ to the propositional language of \luk\ logic that allows us to  apply  the latter and the well-developed theory of MV-algebras directly to probabilistic reasoning. More precisely, leveraging on such translation map, we will show proof-theoretical properties for $\FPL$ and introduce a class of algebras with respect to which $\FPL$ will be proved to be locally sound and complete. Finally, we will apply these previous results to investigate what we called ``probabilistic unification problem''. In this respect, we will prove that Ghilardi's algebraic view on unification can be extended to our case and, on par with the \luk\ propositional case, we show that probabilistic unification is of nullary type.
\end{abstract}

\section{Introduction and motivation}\label{sec1}
Identifying probability theory as part of logic surely is one of the main conceptual contributions and groundbreaking ideas that George Boole reported in the introduction of his seminal work \cite[\S1]{Boole}:
 
\begin{quotation} 
{\em  The design of the following treatise is to investigate the fundamental laws of those operations of the mind by which reasoning is performed; to give expression to them in the symbolical language of a Calculus, and upon this foundation to establish the science of Logic and construct its method; to make that method itself the basis of a general method for the application of the mathematical doctrine of Probabilities.}
\end{quotation}
In the quotation above, Boole recognizes probability theory as a subject that subsumes a type of reasoning that can be handled with the tools of (mathematical) logic  and the symbolical language of algebra. That is the reason why, nowadays, we tend to distinguish {\em probability calculus} and {\em probability logic} as two complementary, yet deeply interconnected, aspects of probability theory.  

		More recent times than those in which Boole published his work, have seen a flourishing of formal methods and logical approaches to deal with probability reasoning. Among them, it is worth recalling the model theoretical approach mainly developed by Keisler \cite{Keis76} and Hoover \cite{Hoo}; the more artificial intelligence oriented perspective initiated by Fagin, Halpern and Megiddo in \cite{FHM90}  and the one put forward by H\'ajek, Esteva and Godo in \cite{HGE}. In the latter, that we will mainly follow here,  probability is understood as a physical variable and it is modeled by a modal operator $P$ added to the language of \luk\ logic; formulas of the form $P(\varphi)$---for $\varphi$ any classical formula--- read as ``$\varphi$ is probable''. Interestingly, the logic of \cite{FHM90} and a slight variant of H\'ajek, Esteva and Godo's logic  have been shown to be syntactically interdefinable, and hence equivalent, in the recent \cite{BCN}.
	
	The equivalent algebraic semantics, in the sense of \cite{BP89}, of \luk\ logic is the class of MV-algebras, an algebraic variety whose generic structure  is defined, like probability logics, on the real unit interval $[0,1]$. The truth-value of a formula like $P(\varphi)$, once evaluated to $[0,1]$ by a \luk\ evaluation, is  the probability of $\varphi$.

The rationale behind what we briefly recalled in the above paragraph is that, although the non-negligible differences that distinguish \luk\ logic and probability logic, one can leverage on their similarities and expand the former by the extra operator $P$ and axiomatizing it in such a way to formalize probability reasoning.

In this paper we will be concerned with an extension of H\'ajek, Esteva and Godo's logic firstly axiomatized in \cite{FG07}, denoted by $\FPL$ and that has been recently proved (cf. \cite{F21}) to be the logic of  {\em state theory}: a generalization of probability theory for uncertain quantification on \luk\ sentences, introduced by Mundici in \cite{MuStates}. In $\FPL$, \luk\ logic plays a twofold role: it is the {\em inner} logic that represents the formulas that fall under the scope of the modality $P$ (i.e., {\em events}) and it is also the {\em outer} logic that reasons on complex probabilistic modal formulas.

More in detail, we will show that, roughly speaking, the modal expansion leading to the logic $\FPL$ is not needed to formalize probabilistic reasoning within \luk\ calculus. Indeed,  the categorical duality  between rational polyhedra and finitely presented MV-algebras put forward in \cite{MS13} will allow us to encode local, finitary, probabilistic information, described by the convex rational polyhedra being the geometric interpretation of de Finetti's coherence criterion (a foundation of probability theory) within \luk\ logic itself. By doing so, we will also consider a translation map from the modal (outer) language of $\FPL$ to the  propositional language of \luk\ logic that preserves, under basic needed assumptions, all theorems and deductions of $\FPL$.

The idea of translating the modal language of probability logics to the propositional \luk\ language is not actually new, and it has been used, for instance, to prove soundness and completeness for H\'ajek, Esteva and Godo's logic with respect to probability models \cite{HGE}. In more abstract terms,  the papers \cite{FGM11b} and \cite{CN} discuss the effect of such translation in general and set the minimal requirement for uncertainty logics to be complete w.r.t. their standard semantics. It is also worth recalling that in \cite{BCN} a similar translation has been used to introduce a proof-calculus for probabilistic reasoning. However, as observed in \cite{FG07} and \cite{F08}, the usual technique that allows to apply such a translation to  prove, for instance, soundness and completeness of probability logic, does not well-behave if the inner-logic, i.e., the logic used to represent events, is not locally finite, like in our case with \luk\ calculus. This is the reason why, in this paper, we need to come up with a new way of translating deductions that allows to handle the non-local finiteness of \luk\ logic and that relies, as already recalled above, on de Finetti's foundational work on coherence and its geometric characterization in terms of finite dimensional polytopes. 

In the present paper, besides detailing what is the effect of such translation to $\FPL$ and showing what results can be proved by its application, leveraging on the categorical duality between rational polyhedron and finitely presented MV-algebras, we will also investigate it in purely algebraic terms identifying a class of MV-algebras that form an algebraic semantics for $\FPL$. These algebras, that will be called {\em coherent}, form a proper subclass of finitely presented and projective MV-algebras. It is worth pointing out that coherent MV-algebras do not provide an equivalent algebraic semantics for $\FPL$ and indeed the problem of establishing the algebraizability of $\FPL$ is still open. 

In the last section of this paper, we will apply the results obtained by the aforementioned translation map and the algebraic properties of coherent MV-algebras to investigate what we call the {\em probabilistic unification problem} 
by exploiting the  key idea of treating the atomic modal formulas of the form $P(\varphi)$ as probabilistic variables. In this sense, and in complete analogy with the usual unification problem for algebraizable logics, unification problems can be easily presented as follows. Given a set of pairs of probabilistic modal formulas $\{(\Phi_i, \Psi_i)\mid i=1,\ldots, m\}$, find, if it exists, a {\em probabilistic substitution} $\sigma$ that maps modal subformulas of the form $P(\varphi_j)$ to (compound) probabilistic terms such that, for all $i=1,\ldots, m$ the identities $\sigma(\Phi_i)=\sigma(\Psi_i)$ hold in $\FPL$. In order to approach this kind of unification, we will first prove that Ghilardi's algebraic approach to  unification problems for algebraizable logics \cite{G97} has an analogous formulation also in our non-algebraizable setting. Secondly, we will show how the pathological example of Marra and Spada \cite{MS13} witnessing that the unification type of \luk\ logic is nullary, can be adapted to the case of $\FPL$ to prove that the probabilistic unification problem is nullary as well. That is, there are probabilistic unification problems with a co-final chain of unifiers of generality order type $\omega$.

The present paper is structured as follows: next section will recall needed notions and results concerning \luk\ logic and MV-algebras (Subsection \ref{sec21}) and in particular free and finitely presented MV-algebras (Subsection \ref{sec22}). In Subsection \ref{sec23} we will present the probability logic $\FPL$ and its semantics based on states. In the same Section \ref{sec2}, we will also present some new results on finitely presented MV-algebras and $\FPL$ that will be useful for what follows.
 A basic introduction to de Finetti's coherence and its geometry will be the subject of Section \ref{sec:cohe}, while in Section \ref{sec4} we will go back to investigate the logic $\FPL$ and the translation map that allows to regard its modal language at the propositional ground. In particular, we will show how to locally reduce, modulo the aforementioned translation, the entailment relation of $\FPL$ to that of \luk\ logic.  As consequences, we show the decidability of the deducibility relation of $\FPL$ and we obtain a local deduction theorem for the probabilistic logic.
 Coherent MV-algebras are defined in Section \ref{sec5} where, besides showing them to be special cases of projective structures (see Subsection \ref{sec51}), we will also prove, in Subsection \ref{sec62}, a local soundness and completeness theorem for $\FPL$ w.r.t. those algebras. In Section \ref{sec6} we will present what we called probabilistic unification problem, we prove how Ghilardi's approach can be rephrased in our context and, finally, we will adapt Marra and Spada's pathological example to the probability framework to show that the probabilistic unification is of nullary type.

\section{Preliminaries}\label{sec2}
In the present section we will go through the basic logical and algebraic notions on which the present paper is grounded. Precisely, \luk\ logic and MV-algebras (Subsection \ref{sec21}), finitely generated free MV-algebras and finitely presented MV-algebras (Subsection \ref{sec22}) and the probability logic $\FPL$ with its semantics based on states (Subsection \ref{sec23}). Besides recalling necessary notions and facts, new results on these subjects will be  proved and commented.  
\subsection{\luk\ logic and MV-algebras}\label{sec21}
\luk\ logic, \L\ in symbols, is a non-classical, many-valued calculus that can be axiomatized within a signature having the primitive binary connective $\oplus$ (disjunction),  the unary connective $\neg$ (negation), and the constant $\bot$ (falsum). Formulas, that we will henceforth denote by lower-case Greek letters, are defined by a non-empty set of propositional variables as usual. 
Other useful connectives and constants symbols are definable within the language of \luk\ logic as follows:
\begin{center}
$\top:=\neg\bot$; $\varphi\to \psi:=\neg \varphi\oplus \psi$; $\varphi\vee\psi:=(\varphi\to\psi)\to\psi$; $\varphi\wedge\psi:=\neg(\neg\varphi\vee\neg\psi)$; $\varphi\odot\psi:=\neg(\neg\varphi\oplus\neg\psi)$; $\varphi\leftrightarrow\psi:=(\varphi\to\psi)\wedge(\psi\to\varphi)$.
\end{center}
Furthermore, if $\varphi$ is any formula and $n$ is a positive integer, we will abbreviate 
\begin{center}
$\varphi\odot\ldots\odot\varphi$ ($n$-times) by $\varphi^n$ and\\
$\varphi\oplus\ldots\oplus\varphi$ ($n$-times) by $n\varphi$.
\end{center}
The set of formulas of \luk\ logic, will be henceforth denoted by ${\bf Fm}$, while ${\bf Fm}(k)$ will denote the set of formulas defined upon $k$ propositional variables.

Axioms and rules for \L\ are as follows:
\begin{itemize}
\item[(\L1)] $\varphi\to(\psi\to\varphi)$,
\item[(\L2)] $(\varphi\to \psi)\to((\psi\to\gamma)\to(\varphi\to \gamma))$,
\item[(\L3)] $((\varphi\to \psi)\to\psi)\to((\psi\to \varphi)\to\varphi)$,
\item[(\L4)] $(\neg\varphi\to\neg\psi)\to(\psi\to\varphi)$,
\item[(MP)] From $\varphi,\varphi\to\psi$, deduce $\psi$ ({\em modus ponens} rule).
\end{itemize}
Theorems are defined as usual and we will write $\vdash_{\lu}\varphi$ to denote that $\varphi$ is a theorem of \luk\ logic. In particular, we will say that two formulas $\varphi$ and $\psi$ are {\em logically equivalent} if $\vdash_{\lu}\varphi\leftrightarrow\psi$.

 If $\Gamma=\{\varphi_1,\varphi_2,\ldots,\varphi_\alpha\}$ are countably many (possibly infinitely many) formulas and  $\varphi$ is a formula, $\Gamma\vdash_{\lu}\varphi$ denotes that $\varphi$ can be deduced from the $\varphi_i$'s within \luk\ calculus.

In the statement of the next proposition, and elsewhere in the paper, we will write formulas as $\varphi(x_1,\ldots, x_k)$ whenever we need to highlight the propositional variables occurring in them. Moreover, if $\varphi(x_1,\ldots, x_k)$ is a formula in $k$ variables and $\tau_1,\ldots, \tau_k$ are formulas, we will write $\varphi(\tau_1,\ldots, \tau_k)$ to denote the formula obtained by substituting, for all $i=1,\ldots, k$, the variable $x_i$ by the formula $\tau_i$. 
\begin{proposition}\label{propBasicPropLuk}
The following properties hold for \L:
\begin{itemize}
\item[(SE)] Substitution of equivalents: if $\varphi(x_1,\ldots, x_k)\in {\bf Fm}(k)$, $\psi_1,\gamma_1,\ldots, \psi_k,\gamma_k\in {\bf Fm}$ are such that $\psi_i$ and $\gamma_i$ are logically equivalent for all $i=1,\ldots, k$, then
$$
\vdash_{\lu}\varphi(\psi_1,\ldots,\psi_k)\leftrightarrow\varphi(\gamma_1,\ldots, \gamma_k).
$$
\item[(LDT)] Local deduction theorem: if $\varphi,\psi$ are formulas, then $\varphi\vdash_{\lu}\psi$ iff there exists a positive integer $n$ such that $\vdash_{\lu}\varphi^n\to\psi$.
\end{itemize}
\end{proposition}

\luk\ logic is  algebraizable in the sense of Blok and Pigozzi \cite{BP89}, and its equivalent algebraic semantics is the variety $\mathsf{MV}$ of MV-algebras, as introduced by Chang in \cite{Chang} (see also \cite{CDM} for an exhaustive treatment). 
Due to this, we will not distinguish the propositional language of the logic from the algebraic signature of MV-algebras. 
\begin{definition} 
An {\em MV-algebra} is a system ${\bf A}=(A, \oplus, \neg, \bot)$ of type $(2,1,0)$, where the following properties hold:
\begin{itemize}
\item[(MV1)] $(A,\oplus, \bot)$ is a commutative monoid,
\item[(MV2)] $\neg\neg x=x$,
\item[(MV3)] $\neg(\neg x\oplus y)\oplus y=\neg(\neg y\oplus x)\oplus x$.
\end{itemize} 
\end{definition}

In  every MV-algebra ${\bf A}$  
one can define further operations and constants on $\alg A$ according to the above identification.  MV-algebra are well behaved with respect to their structure theory, see e.g. \cite[\S 1]{CDM}. In particular, congruences correspond to {\em filters}, that is, nonempty subsets of the domain closed under $\odot$ and upwards. Precisely, for every MV-algebra $\alg A$, the congruence lattice $Con(\alg A)$ is isomorphic to the filter lattice $Fil(\alg A)$ by the following maps:

 \begin{eqnarray}
 \theta\in Con(\alg A)&\longmapsto& F_\theta=\{x\in A\mid (x, \top)\in \theta\}\in Fil(\alg A)\mbox{ and }\\\label{eqFilt1}
F\in Fil(\alg A)&\longmapsto&  \theta_F=\{(x, y)\in A\times A\mid (x\to y), (y\to x)\in F\}\in Con(\alg A).\label{eqFilt2}
\end{eqnarray}
Moreover, finitely generated filters are principally generated by the meet of their finitely many generators. Thus, finitely generated congruences are also principal. As usual in an algebraic setting, we blur the distinction between finitely presentable and finitely presented algebras. Therefore we say that 
an MV-algebra is {\em finitely presented} if it is isomorphic to the quotient of a free finitely generated MV-algebra by a finitely generated, and hence principal, congruence $\theta$.

 {\em Valuations} of \luk\ language in an MV-algebra ${\bf A}$ with support $A$, are functions $e$ mapping propositional variables to $A$ and commuting on each connective and constant.
{\em Tautologies} are those formulas that evaluate to $\top:=\neg \bot$ under every valuation.

Chang's completeness theorem shows that \L\ is sound and complete with respect to the so called {\em standard} MV-algebra, the structure $[0,1]_{MV}=([0,1], \oplus, \neg, 0)$ where, for all  $x,y\in [0,1]$, $x\oplus y=\min\{1, x+y\}$ and $\neg x=1-x$. In algebraic terms, $[0,1]_{MV}$ generates the variety $\mathsf{MV}$. Indeed, the standard MV-algebra generates the variety of MV-algebras as both a variety and a quasivariety \cite[Corollary 7.2]{GiMu}. This fact allows a remarkable characterization of free algebras in $\mathsf{MV}$ that we will recall in the next subsection. 

\subsection{Free and finitely presented MV-algebras}\label{sec22}
MV-algebras form a variety, i.e., an equational class. Therefore, by Birkhoff Theorem \cite{BS}, free MV-algebras exist in $\mathsf{MV}$. If $X$ is any set we will denote by $\free(X)$ the MV-algebra freely generated by $X$.  In the rest of this paper, we will always assume $X$ to be finite and non-empty. 
\begin{remark}\label{remFree}
Every free algebra $\free(X)$ in a variety $\mathsf{V}$ is characterized by the well-known universal property: for every algebra ${\bf A}\in \mathsf{V}$ and every function $f:X\to A$, there exists a unique homomorphism $h_f:\free(X)\to{\bf A}$ that extends $f$. Thus, whenever $X$ has finite cardinality, say $|X|=k$, we will denote by $\free(k)$ the free algebra on variables $\{x_1,\ldots, x_k\}$. The map $f$ that bijectively maps the elements of $X$ to the set of variables $\{x_1,\ldots, x_k\}$ gives an isomorphism between $\free(X)$ and $\free(k)$. Such an identification of elements of a set $X$ to variables from a set of $|X|$ elements will be largely used along this paper.
\end{remark}
By a common universal algebraic argument and since $[0,1]_{MV}$ generates $\mathsf{MV}$, for every finite $k$, $\free(k)$ is isomorphic to the subalgebra of the MV-algebra $[0,1]^{[0,1]^k}$ of functions from $[0,1]^k$ to $[0,1]$ generated by the projection maps, and operations defined pointwise by those on the standard MV-algebra. Furthermore, recall that for all finite $k$, $\free(k)$ is, up to isomorphism, the Lindenbaum-Tarski algebra ${\bf L}(k)$ of \luk\ logic on formulas from ${\bf Fm}(k)$. 

 McNaughton Theorem provides us with a clear geometric characterization of finitely generated free MV-algebras and hence  a {\em functional representation} of (equivalence classes, modulo logical equivalence, of) formulas of each Lindenbaum-Tarski algebra ${\bf L}(k)$. Recall that a function $f:[0,1]^k\to[0,1]$ is named a {\em McNaughton function} if it is continuous, piecewise linear and such that each piece has  integer coefficients. For every positive integer $k$, $\mathcal{M}(k)$ denotes the MV-algebra  of McNaughton functions on $[0,1]^k$ with pointwise operations as in $[0,1]_{MV}$.
\begin{theorem}[\cite{McN,MundConstructive}]\label{Them:ISO1}
For every positive integer $k$, $\free(k)$, ${\bf L}(k)$ and the algebra $\mathcal{M}(k)$  are isomorphic.  
\end{theorem}

The above theorem hence tells us that, for every formula $\varphi\in {\bf Fm}(k)$, its equivalence class $[\varphi]$ in ${\bf L}(k)$ can be regarded, up to isomorphism, as a McNaughton function $f_\varphi:[0,1]^k\to[0,1]$. Conversely, for every McNaughton function $f\in \mathcal{M}(k)$ there is a (not unique) formula $\varphi$, such that $[\varphi]$ is mapped to $f$ by the isomorphism between $\mathcal{M}(k)$ and ${\bf L}(k)$.

The next proposition recalls known facts concerning rational polyhedra and {\em onesets} of McNaughton functions. Remind that  a  {\em (rational) polytope} of $\mathbb{R}^k$ is the convex hull of finitely many points of $\mathbb{R}^k$ ($\mathbb{Q}^k$ respectively); a (rational) polyhedron is  a finite union of (rational) polytopes. Moreover, for every $k$ and for every McNaughton function $f\in \mathcal{M}(k)$, the {\em oneset} of $f$ is ${\rm O}(f)=\{x\in [0,1]^k\mid f(x)=1\}$. 
\begin{proposition}[{\cite[Theorem 3.20]{MuAdvanced}}]\label{PropPolyForm}
{\rm (1)} For every rational polyhedron $\mathscr{P}\subseteq[0,1]^k$, there exists  $\chi_\mathscr{P}\in {\bf Fm}(k)$  such that the McNaughton function $f_{\chi_\mathscr{P}}$ satisfies ${\rm O}(f)=\mathscr{P}$.

{\rm (2)} For every $\varphi\in {\bf Fm}(k)$, ${\rm O}(f_\varphi)$ is a rational polyhedron of $[0,1]^k$. 

{\rm (3)} For every pair of formulas $\varphi,\psi\in {\bf Fm}(k)$, $\varphi\vdash_{\lu}\psi$ iff ${\rm O}(f_\varphi)\subseteq {\rm O}(f_\psi)$ as rational polyhedra.
\end{proposition}

The proposition above allows for a geometrical representation of principal filters and congruences of free finitely generated MV-algebras. Indeed,
the principal filter $F$  of $\free(k)$ generated by $f$, and hence the principal congruence $\theta_F$, correspond by the above proposition to the rational polyhedron ${\rm O}(f)$. Vice versa, given a rational polyhedron $\mathscr{P}\subseteq[0,1]^k$, $F_{\mathscr{P}}=\{g\in \free(k)\mid {\rm O}(g)\supseteq \mathscr{P}\}$ is the filter of $\free(k)$ principally generated by any $f$ such that ${\rm O}(f)=\mathscr{P}$.

Therefore, finitely generated quotients of finitely generated free algebras correspond to  rational polyhedra. In particular, an MV-algebra is finitely presented iff it is  isomorphic to an algebra $\mathcal{M}(\mathscr{P})$ of McNaughton functions over a cube $[0,1]^k$, restricted to a rational polyhedron $\mathscr{P}$ \cite[Theorem 6.3]{MuAdvanced}.  
Thus, we will adopt the following notation.
\begin{notation}
Let $\theta$ be a finitely generated congruence of a finitely generated free MV-algebra, say, $\free(k)$ and let $\mathscr{P}$ be the rational polyhedron of $[0,1]^k$ corresponding to $\theta$. Then, we will henceforth denote the finitely presented MV-algebra $\free(k)/\theta$ by $\free(k)/\mathscr{P}$ without danger of confusion.  
\end{notation}

The last result we will prove in this subsection is meant to extend \cite[Theorem 6.3]{MuAdvanced} to finitely presented (not necessarily free) MV-algebras. First, we need the following.

\begin{proposition}\label{propFOIntersection}
Let $\mathscr{P}, \mathscr{Q}\subseteq [0,1]^k$ be rational polyhedra. Let $\hat{\theta}_\mathscr{Q}$ be the congruence of $\free(k)/\mathscr{P}$ generated by the pairs $([a]_{\mathscr{P}}, [b]_{\mathscr{P}})$ such that $(a,b)$ is a generator of (the congruence associated to) $\mathscr{Q}$. Then, 
$$
(\free(k)/\mathscr{P})/\hat{\theta}_{\mathscr{Q}}\cong\free(k)/(\mathscr{P}\cap\mathscr{Q}).
$$
\end{proposition}
\begin{proof}
From \cite[Theorem 6.3]{MuAdvanced}, $\free(k)/\mathscr{P}\cong \mathcal{M}(\mathscr{P})$ via an isomorphism $\iota$ sending, for every McNaughton function $f:[0,1]^k\to[0,1]$ the equivalence class $[f]_\mathscr{P}$ in $\free(k)/\mathscr{P}$ to the restriction $f_\rest$ of $f$ to $\mathscr{P}$. 
Now, since finitely generated congruences and filters are principal in MV-algebras, also the filter associated to $\scr{Q}$ is principal, and therefore generated by some $g\in \free(k)$, such that ${\rm O}(g)=\mathscr{Q}$. 
Let now $\hat{\mathscr{Q}}$ be the congruence of $\mathcal{M}(\mathscr{P})$ corresponding to $\hat{\theta}_\mathscr{Q}$ via $\iota$. Thus, its associated filter $F$ is generated by $g_\rest\in \mathcal{M}(\mathscr{P})$.

\begin{claim}\label{claim1Rest}
For every McNaughton function $f:[0,1]^k\to[0,1]$, $f\rest\in F$ iff ${\rm O}(g_\rest)\subseteq{\rm O}(f_\rest)$.
\end{claim}
Indeed, if $f_\rest\in F$, then there exists $n$ such that $(g_\rest)^n\leq f_\rest$. Therefore, if $g_\rest(x)=1$, $(g_\rest)^n(x)=1$ and hence $x\in {\rm O}(f_\rest)$. Conversely, assume that ${\rm O}(g_\rest)\subseteq{\rm O}(f_\rest)$, then a slight modification of \cite[Lemma 2.2(i)]{Kroupa12} shows that for some $n$, $(g_\rest)^n\leq f_\rest$ setting the claim.

Now, notice that
\begin{equation}\label{eqRestInt}
{\rm O}(g_\rest)={\rm O}(g)\cap \mathscr{P}=\mathscr{Q}\cap\mathscr{P}. 
\end{equation}
Therefore,  Claim \ref{claim1Rest} becomes 
\begin{equation}\label{eqRestInt2}
f_\rest\in F\mbox{ iff }\mathscr{Q}\cap \mathscr{P}\subseteq {\rm O}(f)\cap \mathscr{P}. 
\end{equation}
Now, we prove that $\mathcal{M}(\mathscr{P})/\hat{\mathscr{Q}}\cong \mathcal{M}(\mathscr{Q}\cap\mathscr{P})$ via the map 
$$
\lambda: [f_\rest]_{\hat{\mathscr{Q}}}\mapsto f_{\upharpoonright_{\mathscr{Q}\cap\mathscr{P}}}.
$$
Let us start showing that $\lambda$ is well-defined, and take $[f_\rest]_{\hat{\mathscr{Q}}}= [h_\rest]_{\hat{\mathscr{Q}}}$. By (\ref{eqFilt2}) this holds iff $f_\rest \to h_\rest\in F$ and $h_\rest\to f_\rest\in F$ iff, by the definition of operations on quotients, $(f\to h)_\rest\in F$ and $(h\to f)_\rest\in F$. By (\ref{eqRestInt2}) the latter is the case iff $\mathscr{Q}\cap\mathscr{P}\subseteq {\rm O}(f\to h)\cap\mathscr{P}$ and $\mathscr{Q}\cap\mathscr{P}\subseteq {\rm O}(h\to f)\cap\mathscr{P}$ iff for all $x\in \mathscr{Q}\cap\mathscr{P}$, $(f\to h)(x)=f(x)\to h(x)=1$ and $(h\to f)(x)=h(x)\to f(x)=1$ iff for all $x\in \mathscr{Q}\cap\mathscr{P}$, $f(x)\leq h(x)$ and $h(x)\leq f(x)$ iff $f_{\upharpoonright_{\mathscr{Q}\cap\mathscr{P}}}=h_{\upharpoonright_{\mathscr{Q}\cap\mathscr{P}}}$ that is to say, $\lambda([f_\rest]_{\hat{\mathscr{Q}}})=\lambda([h_\rest]_{\hat{\mathscr{Q}}})$.

Notice that the above argument (read backwards) also shows that $\lambda$ is injective. Surjectivity is also clear. To finish the proof, we hence need to prove that $\lambda$ is a homomorphism. Let us show that $\lambda$ commutes w.r.t. $\oplus$, the other cases will follow by a similar argument. Let us notice that $\lambda([f_\rest]_{\hat{\mathscr{Q}}}\oplus [h_\rest]_{\hat{\mathscr{Q}}})=\lambda([f_\rest\oplus h_\rest]_{\hat{\mathscr{Q}}})=\lambda([(f\oplus h)_\rest]_{\hat{\mathscr{Q}}})=(f\oplus h)_{\upharpoonright_{\mathscr{Q}\cap\mathscr{P}}}=f_{\upharpoonright_{\mathscr{Q}\cap\mathscr{P}}}\oplus h_{\upharpoonright_{\mathscr{Q}\cap\mathscr{P}}}=\lambda([f_\rest]_{\hat{\mathscr{Q}}})\oplus\lambda([h_\rest]_{\hat{\mathscr{Q}}})$. 
\end{proof}
The previous result allows to prove the following general fact.
\begin{corollary}
Every finitely generated quotient of a finitely presented MV-algebra is finitely presented. 
\end{corollary}
\begin{proof}
Let ${\bf A}=\free(k)/\mathscr{P}$ be finitely presented and let $\theta$ be a finitely generated congruence of ${\bf A}$. Let $[g]_\mathscr{P}$ be a generator of the filter associated to $\theta$ with $g\in \free(k)$. Let now $\mathscr{Q}={\rm O}(g)$. Then we can apply Proposition \ref{propFOIntersection} and get that ${\bf A}/\theta\cong\free(k)/(\mathscr{P}\cap\mathscr{Q})$. Therefore, since the intersection of polyhedra is a polyhedron, ${\bf A}/\theta$ is finitely presented.
\end{proof}

\begin{notation}[Events]
Starting from next subsection, we will be concerned with uncertainty quantification on MV-algebras and, in particular, on free MV-algebras. Adhering to a standard notation, formulas of \luk\ language will be hence called {\em events}. Moreover, thanks to Theorem \ref{Them:ISO1}, we will sometimes identify a formula $\varphi$ with its associated McNaughton function $f_\varphi$. With no danger of confusion, we will refer to both these expressions as to the {\em event $\varphi$} or the {\em event $f_\varphi$}.
\end{notation}

\subsection{The logic $\FPL$ and states on MV-algebras}\label{sec23}

The language of $\FPL$ is obtained by expanding that of \luk\ logic (recall Section~\ref{sec21}) by a unary modality $P$. The set of formulas, denoted by ${\bf PFm}$, is made of the following  two classes:
\vspace{.2cm}

\noindent (EF): the set of {\em event} formulas which contains all formulas of \luk\ language; these formulas will be denoted, as above, by lowercase Greek letters $\varphi,\psi,\ldots$ with possible subscripts; 
\vspace{.1cm}

\noindent (MF): the set of {\em modal} formulas which contains atomic modal formulas, i.e.,  expressions of the form $P(\varphi)$ for every event formula $\varphi$, the constants $\top$ and $\bot$ and which is closed under the connectives of \luk\ language. Modal formulas will be denoted by uppercase Greek letters $\Phi,\Psi,\ldots$ with possible subscripts.
\vspace{.2cm}

Notice that modal formulas in ${\bf PFm}$ are just MV-terms written using atomic modal formulas (thought) as variables. That is, every (compound) modal formula $\Phi$ is of the form $t[P(\varphi_1),\ldots, P(\varphi_k)]$ where $t[x_1,\ldots, x_k]$ is an MV-term on $k$ variables and $P(\varphi_1),\ldots, P(\varphi_k)$ are atomic modal formulas. Indeed, modal formulas of $\FPL$ can be regarded as having two layers: an {\em inner layer} and an {\em outer layer}. The former concerns with the inner atomic modal formulas like the above $P(\varphi_1),\ldots, P(\varphi_k)$ and it is about the probabilistic uncertainty on  events $\varphi_1,\ldots,\varphi_k$;  the latter allows one to  combine the inner formulas $P(\varphi_i)$'s  by means of \luk\ connectives. By doing so, we are able to express properties of atomic probabilistic formulas. For instance the formula (P3) below expresses the finite additivity law. 
 In what follows we shall write $Var(\Phi_1,\ldots,\Phi_m)$ for the set of (inner) \luk\ variables of the event formulas occurring in the compound modal formulas $\Phi_i$'s.

Axioms and rules of FP$(\L,\L)$ are as follows:
\vspace{.2cm}

\noindent(E\L): all axioms and rules of \luk\ calculus for  event formulas;
\vspace{.2cm}

\noindent (M\L): all axioms and rules of \luk\ calculus for modal formulas;

\vspace{.2cm}

\noindent (P): the following axioms and rules specific for the modality $P$:

\begin{enumerate}
\item[(P1)] $\neg P(\varphi)\leftrightarrow P(\neg\varphi)$;
\item[(P2)] $P(\varphi\to\psi)\to(P(\varphi)\to P(\psi))$;
\item[(P3)] $P(\varphi\oplus\psi)\leftrightarrow[(P(\varphi)\to P(\varphi\odot\psi))\to P(\psi)]$;
\item[(N)] From $\varphi$ derive $P(\varphi)$ ({\em necessitation} rule). 
\end{enumerate}
The notion of {\em proof} is defined as usual and, for every modal formula $\Phi$, we will henceforth write $\vdash_{FP}\Phi$ to denote that $\Phi$ is a theorem of FP$(\L,\L)$. As in the \luk\ case, if $\Gamma$ is a countable (possibly infinite) set of modal formulas and $\Phi$ is a modal formula, we write $\Gamma\vdash_{FP}\Phi$ to denote that $\Phi$ is provable from $\Gamma$ in $\FPL$.

The next two propositions show that the logic $\FPL$ enjoys the {\em substitution of equivalents}  for both the inner and the outer layers.
\begin{proposition}\label{propSubstitution}
Let $\Phi=t[P(\varphi_1),\ldots,P(\varphi_k)]$ be a formula in $\pfm$ and let $\varphi_1',\ldots, \varphi_k'\in {\bf Fm}$ be 
such that, for all $i=1,\ldots, k$, $\vdash_{\lu}\varphi_i\leftrightarrow\varphi_i'$. If $\Phi'=t[P(\varphi_1'),\ldots,P(\varphi_k')]$, then $\vdash_{FP}\Phi\leftrightarrow\Phi'$. 
\end{proposition}
\begin{proof}
By Proposition \ref{propBasicPropLuk}, \luk\ logic satisfies the substitution of equivalents. Then, it is enough to prove the claim for $\Phi=P(\varphi)$. Let hence $\varphi'$ be a \luk\ formula such that $\vdash_{\lu}\varphi\leftrightarrow\varphi'$. Then, in particular, $\vdash_{\lu}\varphi\to\varphi'$ and $\vdash_{\lu}\varphi'\to\varphi$. From the former, by a step of necessitation (N), we obtain that $\vdash_{FP}P(\varphi\to\varphi')$ and thus, by the axiom (P2), plus modus ponens, we get $\vdash_{FP}P(\varphi)\to P(\varphi')$. Similarly, from $\vdash_{\lu}\varphi'\to\varphi$, we obtain that $\vdash_{FP}P(\varphi')\to P(\varphi)$ and hence the claim is settled. 
\end{proof}

The next result immediately follows from the fact that \luk\ axioms and rules hold for modal formulas.
\begin{proposition}\label{propSubstitution2}
Let $\Phi=t[P(\varphi_1),\ldots,P(\varphi_k)]$ be a formula in $\pfm$ and let $t'$ be a \luk\ term in $k$ variables that is logically equivalent (in \luk\ logic) to $t$. If $\Phi'=t'[P(\varphi_1),\ldots,P(\varphi_k)]$, then $\vdash_{FP}\Phi\leftrightarrow\Phi'$. 
\end{proposition}

The most natural semantics for $\FPL$ is the one provided by {\em states of MV-algebras}. \begin{definition}[\cite{MuStates}]\label{defStates}
For every MV-algebra ${\bf A}$, a {\em state of ${\bf A}$} is a function $s:A\to[0,1]$  satisfying
\begin{itemize}
\item[(s1)] $s(\top)=1$ ({\em normalization}) and
\item[(s2)] $s(a\oplus b)=s(a)+s(b)$ for all $a, b\in A$ such that $a\odot b=\bot$ ({\em finite additivity}).
\end{itemize}
\end{definition}

States  are finitely additive probability functions if the MV-algebra ${\bf A}$ is in particular a Boolean algebra (i.e., it satisfies $x\vee\neg x=\top$). Furthermore, every homomorphism of an MV-algebra ${\bf A}$ to $[0,1]_{MV}$ is a state. More in general, every state of an MV-algebra ${\bf A}$ belongs to the topological closure, in the product space $[0,1]^A$, of the convex hull of homomorphisms of ${\bf A}$ to $[0,1]$, see \cite{MuStates,MuAdvanced,KF15} for further details.

Notice that the axioms and rules of $\FPL$ are enough to syntactically prove that the modality $P$ satisfies the basic properties of states. For instance, $\top\to P(\top)$ follows by the necessitation rule and $P(\top)\to \top$ also follows since $\varphi\to\top$ is a theorem of \luk\ logic. Therefore $P(\top)\leftrightarrow \top$, that corresponds to the above (s1), can be proved in $\FPL$. 
Instantiating  (P1) with $\varphi=\top$, one obtains $\neg P(\top)\leftrightarrow P(\neg\top)$. Since $P(\top)\leftrightarrow\top$ and $\neg\top\leftrightarrow\bot$, we get that $\bot\leftrightarrow P(\bot)$ which  reads ``the probability of a contradiction is zero''. Finally, the finite additivity (s2) of P is proved as follows: let $\varphi$ and $\psi$ be such that $\varphi\odot \psi\leftrightarrow\bot$ is a theorem. Then, by necessitation $P(\varphi\odot\psi)\leftrightarrow \bot$ is a theorem as well, and substituting $P(\varphi\odot\psi)$ by $\bot$ in (P3), one has $P(\varphi\oplus\psi)\leftrightarrow[(P(\varphi)\to \bot)\to P(\psi)]$. Now, $P(\varphi)\to \bot$ is equivalent, in \luk\ logic, to $\neg P(\varphi)$, thus $(P(\varphi)\to \bot)\to P(\psi)$ is  $\neg P(\varphi)\to P(\psi)$ that equals $P(\varphi)\oplus P(\psi)$. Hence, from $\vdash_{FP}\varphi\odot\psi\leftrightarrow\bot$, we  infer $\vdash_{FP}P(\varphi\oplus\psi)\leftrightarrow P(\varphi)\oplus P(\psi)$. 

Now, let $\Phi=t[P(\varphi_1),\ldots, P(\varphi_k)]$ be a modal formula of $\pfm$ and assume that $n=|Var(\Phi)|$. Then, if $s$ is a state of $\free(n)$, we can evaluate $\Phi$ in the standard MV-algebra $[0,1]_{MV}$ by $s$ in the following way: 
$$
t^{[0,1]_{MV}}[s(f_{\varphi_1}), \ldots, s(f_{\varphi_k})].
$$
\begin{definition}
	For every formula $\Phi$ in $n$ variables and for every state $s$ of $\free(n)$, we will write $s\models \Phi$ if $t^{[0,1]_{MV}}[s(f_{\varphi_1}), \ldots, s(f_{\varphi_k})]=1$. 
\end{definition}

In \cite[Theorem 4.2]{F21}, the logic $\FPL$ is shown to be complete with respect to states, that is to say, for every modal formula $\Phi$ with $n=|Var(\Phi)|$, if $\vdash_{FP}\Phi$ then for all states $s$ of $\free(n)$, $s\models \Phi$. 
Now we present a slight generalization of this standard completeness theorem, that will turn out to be useful in what follows. We will make use of the strong completeness result of $\FPL$ with respect to {\em hyperstates} shown in \cite{F08}.  For every MV-algebra $\alg A$, an hyperstate of $\alg A$ is a map $s^{*}: A \to[0,1]^{*}$, where $[0,1]^*$ is a nontrivial ultrapower of the real unit interval that satisfies (s1) and (s2) of the above Definition \ref{defStates}. Given a modal formula $\Phi=t[P(\varphi_1),\ldots, P(\varphi_k)]$ of  $\pfm$ with $n=|Var(\Phi)|$, if $s^*$ is a hyperstate of $\free(n)$, we write that $s^*\models \Phi$ if 
$
t^{[0,1]^*_{MV}}[s^*(f_{\varphi_1}), \ldots, s^*(f_{\varphi_k})]=1.
$
The following holds.
\begin{proposition}[{\cite[Theorem 4.8]{F08}}]\label{propHyper}
Let $\Phi,\Psi$ be formulas in $\pfm$ such that $n=|Var(\Phi,\Psi)|$. Then $\Phi\vdash_{FP}\Psi$ iff for every hyperstate $s^*$ of $\free(n)$, $s^*\models \Phi$ implies $s^*\models \Psi$.
\end{proposition}

Also, recall from \cite{CDM} that every MV-chain, i.e., every totally ordered MV-algebra, partially embeds into the standard MV-algebra $[0,1]_{MV}$. This means that for every MV-chain ${\bf A}$ and for every finite subset $X$ of $A$, there exists an injective map $\iota: X\to [0,1]_{MV}$ that preserves all the operations appearing in $X$. That is, for instance, if $x, y, x\oplus y\in X$, then $\iota(x\oplus y)=\iota(x)\oplus\iota(y)$.

The next technical lemma is extracted from the proof of \cite[Theorem 4.2]{F21}.
\begin{lemma}\label{LemmaKey1}
For every finite collection of McNaughton functions $f_1,\ldots, f_t$ of $\free(n)$ and for every hyperstate $s^*: \free(n)\to[0,1]^*$, there exist a finite subset $X$ of $[0,1]^*$, a partial embedding $\lambda$ of $X$ to $[0,1]_{MV}$ and a state $s:\free(n)\to[0,1]$ such that for all $j=1,\ldots, t$, 
$$
s(f_j)=\lambda(s^*(f_j)).
$$
\end{lemma}

We can hence now prove the claimed improvement of the standard completeness of $\FPL$. 
\begin{theorem}\label{thm:compl2}
Let $\Phi,\Psi$ be formulas in $\pfm$ such that $n=|Var(\Phi,\Psi)|$. Then $\Phi\vdash_{FP}\Psi$ iff for every state $s$ of $\free(n)$, $s\models \Phi$ implies $s\models \Psi$. 
\end{theorem}
\begin{proof}
 By Proposition \ref{propHyper}, $\Phi\vdash_{FP}\Psi$ iff there exists a hyperstate $s^*$ of $\free(n)$ such that,  $s^*\models \Phi$ implies $s^*\models \Psi$.  
By Lemma \ref{LemmaKey1} there exist a finite subset $X$ of $[0,1]^*$, a partial embedding $\lambda$ of $X$ to $[0,1]_{MV}$ and a state $s:\free(n)\to[0,1]$ such that for all $\gamma\in \{\varphi_1,\ldots, \varphi_k, \psi_1,\ldots, \psi_l\}$, $s(f_\gamma)=\lambda(s^*(f_\gamma))$. 
Thus,
$$
t^{[0,1]_{MV}}[s(f_{\varphi_1}),\ldots, s(f_{\varphi_k})]=1\mbox{ and } r^{[0,1]_{MV}}[s(f_{\psi_1}),\ldots, s(f_{\psi_l})]<1.
$$
\end{proof}

In light of the above result, the requirements that we made in Proposition \ref{propSubstitution} and Proposition \ref{propSubstitution2} on the fact that the terms $t$ and $t'$ are written in the same number of variables can be shown to be not necessary. Indeed, every formula $\Phi=t[P(\varphi_1),\ldots, P(\varphi_k)]$ can be equivalently rewritten by allowing more atomic modal formulas. More precisely, let $\Phi$ be as above and let $P(\psi_1),\ldots, P(\psi_m)$ be atomic modal formulas not occurring in $\Phi$, and consider the formula 
$$
\Phi'=\Phi\wedge\left(\bigwedge_{j=1}^m [P(\psi_j)\to P(\psi_j)]\right).
$$
Then, one can easily prove that $\vdash_{FP}\Phi\leftrightarrow\Phi'$. Indeed, assuming w.l.o.g. that $n$ is the number of propositional variables occurring in $\Phi'$, by Theorem \ref{thm:compl2}, $\vdash_{FP}\Phi\leftrightarrow\Phi'$ iff, for all state $s$ of $\free(n)$, $s\models \Phi\leftrightarrow\Phi'$, that is to say, 
$$
t^{[0,1]_{MV}}[s(f_{\varphi_1}),\ldots, s(f_{\varphi_k})]=\min\left\{t^{[0,1]_{MV}}[s(f_{\varphi_1}),\ldots, s(f_{\varphi_k})], \bigwedge_{j=1}^m(s(f_{\psi_j})\to s(f_{\psi_j}))\right\}. 
$$
The latter equality is trivially true because, for all $j=1,\ldots, m$ and for all state $s$, $s(f_{\psi_j})\to s(f_{\psi_j})=1$.

Thus, in particular, we immediately get the following proposition that will be helpful in  the next results. 

\begin{proposition}\label{propNumVariables}
Let $\Phi=t[(P(\varphi_1),\ldots,P(\varphi_k)]$ and $\Psi=u[P(\psi_1),\ldots, P(\psi_m)]$ be formulas from $\pfm$. Then there exist $\Phi', \Psi'\in \pfm$ that contain all the atomic modal formulas $P(\varphi_1),\ldots,P(\varphi_k)$, $P(\psi_1),\ldots, P(\psi_m)$ as subformulas, such that $\vdash_{FP}\Phi\leftrightarrow\Phi'$ and $\vdash_{FP}\Psi\leftrightarrow\Psi'$.
\end{proposition}

\section{Coherence, coherent sets and their geometry}\label{sec:cohe}
States of MV-algebras capture the uncertainty quantification of events that are described within the language of \luk\ logic, as probability measures do in the realm of classical logic. Moreover, in analogy with the foundational aspects of classical probability theory, states are the  functions that characterize the natural generalization of de Finetti's no-Dutch-Book criterion \cite{deF1,deF2} to the MV-algebraic realm. 

To see this, recall that de Finetti's foundation of {\em subjective} probability theory is grounded on a betting game between two players (commonly called the {\em bookmaker} and the {\em gambler}) that wage money on the occurrence of some events whose occurrence is unknown. Providing a full detailed presentation of de Finetti's game and its generalizations is out of the scope of the present paper and we urge the interested reader to consult the rich literature on this subject (see e.g., \cite{AGM, deF1, deF2,Mu06,MuAPAL,Paris,Paris1,Weat} and references therein). However, what is important to recall is that, whenever the bookmaker fixes a set of events $\varphi_1,\ldots,\varphi_k$ and {\em selling prices} $\beta_1,\ldots, \beta_k\in [0,1]$ (that is to say, a {\em book} $\beta:\varphi_i\mapsto \beta_i$), those latter are {\em coherent} (and the book $\beta$ is {\em coherent}) if they bar any possible malicious gambler from elaborating a strategy of bets that would let the bookmaker to incur in a {\em sure-loss}. 

What is of key importance for what follows is that, according to de Finetti's Theorem \cite{deF2}, bookmaker's selling prices $\beta_1,\ldots, \beta_k$ are coherent iff they are {\em consistent} with Kolmogorov's axioms of finitely additive probabilities.

\begin{theorem}[de Finetti]
Let $\mathcal{E}=\{\varphi_1,\ldots, \varphi_k\}$ be any finite set of classical events and let $\beta:\varphi_i\mapsto\beta_i$ be a book on them. Then $\beta$ is coherent iff there exists a probability $\mu$ on the Boolean algebra  generated by the $\varphi_i$'s that extends $\beta$. That is to say, $\mu(\varphi_i)=\beta(\varphi_i)$ for all $i=1,\ldots, k$. 
\end{theorem}
De Finetti's coherence criterion is sufficiently robust to extend to the \luk\ realm with essentially no modification and states of MV-algebras characterize this extended notion of coherence \cite{Mu06}.  
\begin{theorem}[Mundici]\label{thm:Mund1}
Let $\mathcal{E}=\{\varphi_1,\ldots,\varphi_k\}$ be a finite set of \luk\ events on $n$ propositional variables. A book $\beta:\varphi_i\mapsto\beta_i$ is coherent iff there exists a state $s$ of $\free(n)$ that extends it.
\end{theorem}

For every set $\mathcal{E}=\{\varphi_1,\ldots, \varphi_k\}$ of events (in $n$ variables), the set of all  coherent books $\beta: \mathcal{E}\to[0,1]$ has a clear geometric representation. Indeed,
 consider the McNaughton functions $f_{\varphi_1},\ldots,f_{\varphi_k}:[0,1]^n\to[0,1]$, the set $\{\langle f_1(x),\ldots, f_k(x)\rangle\mid x \in [0,1]^n\}$ and its convex hull
\begin{equation}\label{eqCoherentSet}
\mathscr{C}_\mathcal{E}=\overline{\rm co}\{\langle f_1(x),\ldots, f_k(x)\rangle\mid x \in [0,1]^n\} \subseteq [0,1]^{k}.
\end{equation}
 Since $\mathscr{C}_\mathcal{E}$ is a set of functions from $\mathcal{E}=\{\varphi_1,\ldots, \varphi_k\}$ to $[0,1]$, we will equivalently regard it as a subset of either $[0,1]^k$ or $[0,1]^{\mathcal{E}}$.

As shown in \cite[Corollary 3.2]{F19},  $\mathscr{C}_\mathcal{E}$ can be defined without considering all elements $x\in [0,1]^n$.  Indeed,  let $\Delta$ be a regular complex\footnote{Recall that a {\em simplicial comple}x $\Delta$ is a nonempty finite set of simplexes such that: the face of each simplex in $\Delta$ belongs to $\Delta$, and for each pair of simplexes $T_1,T_2\in \Delta$  their intersection is either empty, or it coincides with a common face of $T_1$ and $T_2$. A regular complex is a simplicial complex with regular  simplexes (consult \cite{Ewald} for the unexplained notions).} linearizing the McNaughton functions $f_1,\ldots, f_k$. If $\bv_1,\ldots, \bv_t$ are the vertices of $\Delta$,  the above (\ref{eqCoherentSet}) reduces to
$$
\mathscr{C}_\mathcal{E}={\rm co}\{\langle f_1(\bv_j),\ldots, f_k(\bv_j)\rangle\mid j=1,\ldots, t\}. 
$$
By \cite[Corollary 3.2]{F19}, the definition of $\mathscr{C}_\mathcal{E}$ given above does not depend on the specific $\Delta$ we choose to linearize the McNaughton functions $f_{\varphi_i}$'s.

The following example shows how to construct $\mathscr{C}_{\mathcal{E}}$. 
\begin{example}\label{ex:cConstruction}
Consider  two events $\varphi_1,\varphi_2$ whose corresponding McNaughton functions are 
$$
f_{\varphi_1}(x,y)=x\vee y\mbox{ and }f_{\varphi_2}(x,y)=x\oplus y.
$$
 Set the regular complex of $[0,1]^2$ as in Figure \ref{fig1} and notice that it linearizes both $x\vee y$ and $x\oplus y$. The vertices of $\Delta$ are $\bv_1=\langle 0,0\rangle$, $\bv_2=\langle1,0\rangle$, $\bv_3=\langle0,1\rangle$, $\bv_4=\langle1,1\rangle$ and $\bv_5=\langle1/2,1/2\rangle$. 
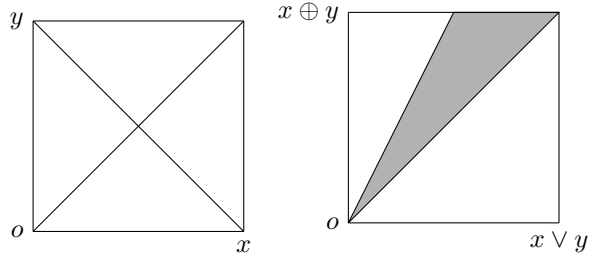
\begin{figure}[H]
\begin{center}
   \begin{tikzpicture}[scale=2.8] 
%
 \coordinate[label=left: {$o$}] (A) at (0,0) ; 
   \coordinate[label=below: {$x$}] (B) at (1,0);
       \coordinate (C) at (1,1);
   \coordinate[label=left: {$y$}](D) at (0,1) ; 

  \draw  (A) -- (B); 
  \draw  (A) -- (D);
  \draw  (B) -- (C);
  \draw  (D) -- (C);
  \draw (B)--(D);
  \draw (A)--(C);
  
  \end{tikzpicture}
    \begin{tikzpicture}[scale=2.8] 
    \filldraw[fill= black!30!white, draw=black]
 (0,0)--(1,1)--(1/2,1)--(0,0);
   
 \coordinate[label=left: {$o$}] (A) at (0,0) ; 
   \coordinate[label=below: {$x\vee y$}](B) at (1,0);
       \coordinate (C) at (1,1);
   \coordinate[label=left: {$x\oplus y$}](D) at (0,1) ; 

  \draw  (A) -- (B); 
  \draw  (A) -- (D);
  \draw  (B) -- (C);
  \draw  (D) -- (C);
  
  \end{tikzpicture}
\end{center}
  \caption{{\small  A triangulation of $[0,1]^2$ linearizing the events in $\mathcal{E}=\{x\vee y,x\oplus y\}$ (picture on the left) and the  set $\mathscr{C}_{\mathcal{E}}$ (on the right).}}\label{fig1}
\end{figure}
One hence obtains
\begin{center}
$
\langle f_{\varphi_1}\langle0,0\rangle, f_{\varphi_2}\langle0,0\rangle\rangle=\langle0 , 0\rangle; \langle f_{\varphi_1}\langle1,0\rangle, f_{\varphi_2}\langle1,0\rangle\rangle=\langle f_{\varphi_1}\langle0,1\rangle, f_{\varphi_2}\langle0,1\rangle\rangle=\langle f_{\varphi_1}\langle1,1\rangle, f_{\varphi_2}\langle1,1\rangle\rangle=\langle1 ,1\rangle,\mbox{ and }\langle f_{\varphi_1}\langle1/2,1/2\rangle, f_{\varphi_2}\langle1/2,1/2\rangle\rangle=\langle1/2 , 1\rangle.
$
\end{center}
Thus,  $\mathscr{C}_{\mathcal{E}}={\rm co}\{\langle0,0\rangle, \langle1,1\rangle, \langle1/2, 1\rangle\}$ 
as represented on the right-hand side of Figure \ref{fig1} in the space whose coordinates are labelled by the events $x\vee y$ and $x\oplus y$.

\end{example} 

Since every book $\beta$ on $\mathcal{E}$ can be regarded as a point $\langle \beta(\varphi_1),\ldots, \beta(\varphi_k)\rangle \in [0,1]^{k}$, the expression $\beta\in \mathscr{C}_{\mathcal{E}}$  makes sense. 
\begin{theorem}\label{thmCoheGeo}
For every set $\mathcal{E}=\{\varphi_1,\dots, \varphi_k\}$ of events in $n$ variables, the following conditions hold:
\begin{enumerate}
\item $\mathscr{C}_{\mathcal{E}}$ is a rational polytope of $[0,1]^k$ and it contains a Boolean point;
\item a book $\beta:\mathcal{E}\to[0,1]$ is coherent iff $\beta\in \mathscr{C}_\mathcal{E}$.
\end{enumerate}
\end{theorem}
\begin{proof}
The first claim of (1) and the claim (2) are \cite[Corollary 3.2]{F19} and \cite[Corollary 5.4]{Mu06} respectively. It is hence left to show that every $\mathscr{C}_\mathcal{E}$ contains a Boolean point, that is a to say a vertex of the cube $[0,1]^k$. This last claim directly follows from the definition of $\mathscr{C}_\mathcal{E}$ together with the fact that each McNaughton function $f:[0,1]^n\to[0,1]$ only takes value in $\{0,1\}$ once restricted to $\{0,1\}^n$. Thus, for every $x\in\{0,1\}^n$, the point of $\mathscr{C}_\mathcal{E}$ of the form $\langle f_{\varphi_1}(x),\ldots,f_{\varphi_k}(x)\rangle$ belongs to $ \{0,1\}^k$ which settles the claim.  
\end{proof}

The next result is hence a corollary of Theorem \ref{thm:compl2}, Theorem \ref{thm:Mund1} and the observation that every state determines coherent books once restricted on finite subsets of its domain, as highlighted in Theorem \ref{thmCoheGeo}. Furthermore, and in light of the above argument, if $\Phi=t[P(\varphi_1),\ldots, P(\varphi_k)]$ is a formula in $\pfm$ and $\beta$ is a coherent book on $\mathcal{E}=\{\varphi_1,\ldots, \varphi_k\}$, we will write 
$$\beta\models \Phi \;\mbox{ iff }\; t^{[0,1]_{MV}}[\beta(\varphi_1),\ldots, \beta(\varphi_k)]=1.$$ 
Clearly, if $\beta\models \Phi$, then for every state $s$ that extends $\beta$, it holds that $s\models \Phi$.

\begin{corollary}\label{ComplCoherBooks}
Let $\Phi=t[P(\varphi_1),\ldots, P(\varphi_k)]$ and $\Psi=r[P(\psi_1),\ldots, P(\psi_m)]$ be formulas from $\pfm$ and let $\mathcal{E}$ be the set of events occurring in $\Phi$ and $\Psi$. 
Then $\Phi\vdash_{FP}\Psi$ iff for all $\beta\in \mathscr{C}_{\mathcal{E}}$ such that $\beta\models \Phi$, then $\beta\models \Psi$. 
\end{corollary}

What we showed so far makes it clear that every finite set of $k$ (\luk) events $\mathcal{E}$ determines  the set $\mathscr{C}_\mathcal{E}$ of all possible coherent books that a bookmaker can define on them. Convex sets of this kind will be formally defined below and called {\em coherent sets}. In the remaining of this section we will present a geometric description of them and prove  some basic properties. 
\begin{definition} \label{defCoheSet}
A convex subset $\mathscr{C}$ of $[0,1]^k$ is said to be a {\em coherent set} if there exists a set of events $\mathcal{E}=\{\varphi_1,\ldots, \varphi_k\}$  such that $\mathscr{C}=\mathscr{C}_\mathcal{E}$.
\end{definition}
Coherent sets are not determined by a unique choice of $\mathcal{E}$. 
For instance, it is easy to see that the same coherent set corresponds to both $\mathcal{E}=\{x\wedge y, x\oplus y\}$ and $\mathcal{E}'=\{x\odot y, x\oplus y\}$. That is to say, $\mathscr{C}_{\mathcal{E}}=\mathscr{C}_{\mathcal{E}'}$. 

Notice also that coherent sets are not compositional. This means that, if $\mathcal{E}$ and $\mathcal{E}'$ are two sets of events, then there is no general geometric construction that allows one to define $\mathscr{C}_{\mathcal{E}\cup\mathcal{E}'}$ from $\mathscr{C}_\mathcal{E}$ and $\mathscr{C}_{\mathcal{E}'}$. However, the {\em projection} of coherent sets to lower dimensional spaces does yield a coherent set.

\begin{proposition}\label{prop:restrictC}
Let $\mathcal{E}=\{\varphi_1,\ldots, \varphi_k\}$ a set of events. Then for every subset $\mathcal{E'}$ of  $\mathcal{E}$, $\mathscr{C}_{\mathcal{E}'}$ coincides with the projection of $\mathscr{C}_\mathcal{E}$ to $[0,1]^{\mathcal{E}'}$. 
\end{proposition}
\begin{proof}
The claim immediately follows observing that the projection of $\mathscr{C}_\mathcal{E}$ to $[0,1]^{\mathcal{E}'}$ consists of all coherent books on $\mathcal{E}'$ and hence it coincides with $\mathscr{C}_{\mathcal{E}'}$. That is to say, for every coherent book $\beta:\mathcal{E}\to[0,1]$, its restriction to $\mathcal{E}'$ is coherent as well.
\end{proof}
The following example gives a geometric intuition of the above result and it also is meant to clarify the non-compositionality of coherent sets.

\begin{example}
Let us consider the set of events $\mathcal{E}=\{\varphi_1,\varphi_2,\varphi_3\}$ in two variables whose McNaughton functions respectively are
$$
f_{\varphi_1}(x,y)=x\oplus y,\; f_{\varphi_2}(x,y)= x\odot y, \;f_{\varphi_3}(x,y)=  x\wedge y. 
$$ In order to describe $\mathscr{C}_\mathcal{E}$ notice that the triangulation on the left-hand side of Figure \ref{fig1} linearizes the $f_{\varphi_i}$'s. Thus, a direct computation shows that 
$$
\mathscr{C}_\mathcal{E}={\rm co}\{\langle0,0,0\rangle, \langle1,0,0\rangle, \langle1,1,1\rangle, \langle1, 1/2, 0\rangle\}
$$
as in the top left of Figure \ref{fig2}. 

The projections of $\mathscr{C}_\mathcal{E}$ to the squares $[0,1]^{(x\oplus y, x\wedge y)}$, $[0,1]^{(x\odot y, x\wedge y)}$ and $[0,1]^{(x\oplus y, x\odot y)}$ are respectively as in the top-right, bottom-left and bottom-right of the same Figure \ref{fig2} and it is immediate to see that they correspond  to the coherent sets of the events $\mathcal{E}'=\{\varphi_1,\varphi_2\}$, $\mathcal{E}''=\{\varphi_2,\varphi_3\}$ and $\mathcal{E}'''=\{\varphi_2,\varphi_3\}$ respectively. 

Finally, notice what we remarked below Definition \ref{defCoheSet}: although $\mathcal{E}'\neq\mathcal{E}'''$, their coherent sets coincide as subset of $[0,1]^2$.   
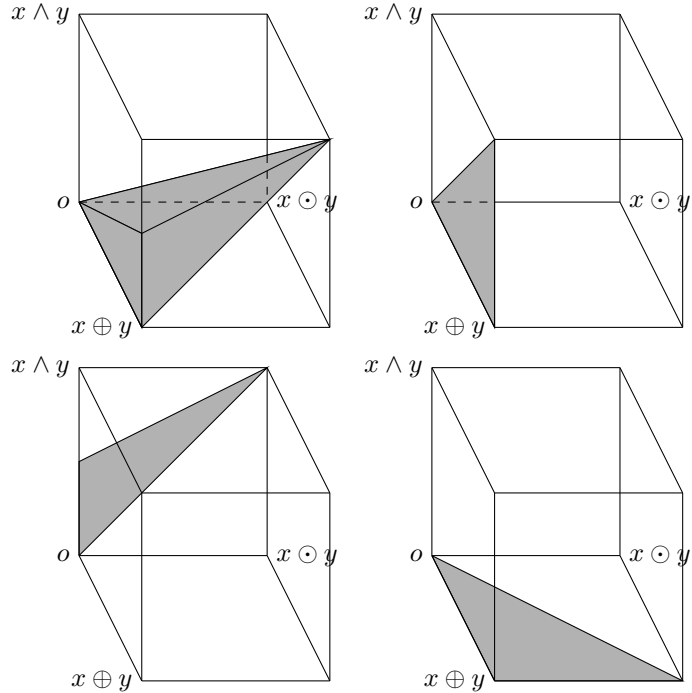
\begin{figure}[H]
\begin{center}
    \begin{tikzpicture}[scale=2.5]

 \coordinate[label=left: {$o$}] (A) at (0,0) ; 
   \coordinate[label=right: {$x\odot y$}](B) at (1,0);
       \coordinate (C) at (1,1);
   \coordinate[label=left: {$x\wedge y$}] (D) at (0,1) ; 
      \coordinate[label=left: {$x\oplus y$}] (E) at (1/3,-2/3);
      \coordinate (F) at (4/3,-2/3);
      \coordinate (G) at (4/3,1/3);
      \coordinate (H) at (1/3,1/3);
      \coordinate (I) at (1/3, -1/6); 
 \coordinate(L) at (1, 1/4);
 
       \filldraw[fill= black!30!white, draw=black]
 (A)--(E)--(G)--(A);
        \filldraw[fill= black!30!white, draw=black]
 (A)--(I)--(G)--(A);
       \filldraw[fill= black!30!white, draw=black]
 (A)--(E)--(I)--(A);

  \draw[dashed]  (A) -- (B); 
  \draw[dashed]  (L) -- (B);
  \draw  (C) -- (L);
  \draw  (D) -- (A);
    \draw  (D) -- (C);
  \draw  (E) -- (F); 
  \draw  (F) -- (G);
  \draw  (G) -- (H);
  \draw  (H) -- (E);
    \draw  (A) -- (E); 
  \draw  (B) -- (F);
  \draw  (C) -- (G);
  \draw  (D) -- (H);

  \end{tikzpicture}
    \begin{tikzpicture}[scale=2.5]

 \coordinate[label=left: {$o$}] (A) at (0,0) ; 
   \coordinate[label=right: {$x\odot y$}](B) at (1,0);
       \coordinate (C) at (1,1);
   \coordinate[label=left: {$x\wedge y$}] (D) at (0,1) ; 
      \coordinate[label=left: {$x\oplus y$}] (E) at (1/3,-2/3);
      \coordinate (F) at (4/3,-2/3);
      \coordinate (G) at (4/3,1/3);
      \coordinate (H) at (1/3,1/3);
      \coordinate (I) at (1/3, -1/6); 
 \coordinate(L) at (1, 1/4);
 \coordinate(N) at (1/3,0);
 
       \filldraw[fill= black!30!white, draw=black]
 (A)--(E)--(H)--(A);
   
  \draw[dashed]  (A) -- (N);
  \draw (N)--(B); 
  \draw  (L) -- (B);
  \draw  (C) -- (L);
  \draw  (D) -- (A);
    \draw  (D) -- (C);
  \draw  (E) -- (F); 
  \draw  (F) -- (G);
  \draw  (G) -- (H);
  \draw  (H) -- (E);
    \draw  (A) -- (E); 
  \draw  (B) -- (F);
  \draw  (C) -- (G);
  \draw  (D) -- (H);

  \end{tikzpicture}

    \begin{tikzpicture}[scale=2.5]

 \coordinate[label=left: {$o$}] (A) at (0,0) ; 
   \coordinate[label=right: {$x\odot y$}](B) at (1,0);
       \coordinate (C) at (1,1);
   \coordinate[label=left: {$x\wedge y$}] (D) at (0,1) ; 
      \coordinate[label=left: {$x\oplus y$}] (E) at (1/3,-2/3);
      \coordinate (F) at (4/3,-2/3);
      \coordinate (G) at (4/3,1/3);
      \coordinate (H) at (1/3,1/3);
      \coordinate (I) at (1/3, -1/6); 
 \coordinate(L) at (1, 1/4);
 \coordinate(M) at (0,1/2);
 
       \filldraw[fill= black!30!white, draw=black]
 (A)--(M)--(C)--(A);

  \draw (A) --(B);
  \draw  (C) -- (B);
  \draw  (D) -- (A);
    \draw  (D) -- (C);
  \draw  (E) -- (F); 
  \draw  (F) -- (G);
  \draw  (G) -- (H);
  \draw  (H) -- (E);
    \draw  (A) -- (E); 
  \draw  (B) -- (F);
  \draw  (C) -- (G);
  \draw  (D) -- (H);

  \end{tikzpicture}
    \begin{tikzpicture}[scale=2.5]

 \coordinate[label=left: {$o$}] (A) at (0,0) ; 
   \coordinate[label=right: {$x\odot y$}](B) at (1,0);
       \coordinate (C) at (1,1);
   \coordinate[label=left: {$x\wedge y$}] (D) at (0,1) ; 
      \coordinate[label=left: {$x\oplus y$}] (E) at (1/3,-2/3);
      \coordinate (F) at (4/3,-2/3);
      \coordinate (G) at (4/3,1/3);
      \coordinate (H) at (1/3,1/3);
      \coordinate (I) at (1/3, -1/6); 
 \coordinate(L) at (1, 1/4);
 
       \filldraw[fill= black!30!white, draw=black]
 (A)--(F)--(E)--(A);
  
  \draw (A) -- (B); 
  \draw  (C) -- (B);
  \draw  (D) -- (A);
    \draw  (D) -- (C);
  \draw  (E) -- (F); 
  \draw  (F) -- (G);
  \draw  (G) -- (H);
  \draw  (H) -- (E);
    \draw  (A) -- (E); 
  \draw  (B) -- (F);
  \draw  (C) -- (G);
  \draw  (D) -- (H);

  \end{tikzpicture}
\end{center}
  \caption{{\small  The coherent set $\mathscr{C}_{\mathcal{E}}$ for $\mathcal{E}=\{x\oplus y, x\odot y, x\wedge y\}$ and its three projections on $[0,1]^{(x\oplus y, x\wedge y)}$, $[0,1]^{(x\odot y, x\wedge y)}$ and $[0,1]^{(x\oplus y, x\odot y)}$ respectively.}}\label{fig2}
\end{figure}
\end{example}

 Let us end this section with the following useful observation.
\begin{remark}\label{rem:Algo} A direct consequence of Proposition \ref{PropPolyForm} is that, for every rational polyhedron $\mathscr{P}\subseteq [0,1]^k$, there exist equi-provable, yet not logically equivalent, formulas such that the onesets of their McNaughton functions are all $\mathscr{P}$.
However,  \cite[Proposition 5.3]{F19} provides an algorithm  that for every rational polyhedron $\mathscr{P}$ determines a specific formula $\chi_\mathscr{P}$ with the above property.  
This argument clearly applies also to coherent sets as the latter are convex rational polyhedra. 
Therefore, for every set of formulas $\mathcal{E}=\{\varphi_1,\ldots, \varphi_k\}$ we will henceforth denote by $\chi_\mathcal{E}$ {\em the} \luk\ formula picked by the above mentioned algorithm such that the oneset of $f_{\chi_\mathcal{E}}$ is $\mathscr{C}_\mathcal{E}$.
\end{remark}

\section{Local reduction of \FPL\ to \luk\ calculus}\label{sec4}
In this section we are going to show how to  encode the language, and locally reduce the deducibility relation, of the probability logic $\FPL$ to propositional \luk\ calculus. Such local reduction is essentially inspired by the previous Corollary \ref{ComplCoherBooks} and it is possible thanks to a translation map from the modal language of $\FPL$ to that of \L.

The idea of translating probability formulas to the propositional language of \luk\ logic is  not new and, in fact, it was the main tool used in \cite{HGE} to prove soundness and completeness for the probability logic on classical events w.r.t. probability spaces. Moreover it has been also adopted in \cite{FG07} for similar purposes (see also \cite{FGM11b, CN} for a more exhaustive discussion). In the more recent paper \cite{BCN} the same idea has been finally employed to present a hypersequent calculus for the probability logic introduced in \cite{HGE} of which $\FPL$ is a proper generalization.

Let us hence start defining the translation map $^\bullet$ from modal formulas of ${\bf PFm}$ to the propositional language ${\bf Fm}$ of \luk\ logic in the following inductive manner:
\begin{itemize}
\item[(T1)] $(\bot)^\bullet=\bot$ and $(\top)^\bullet=\top$;
\item[(T2)] For every atomic modal formula $P(\varphi)$, let $p_\varphi$ be a symbol for a fresh variable in \luk\ language and let $P(\varphi)^\bullet=p_\varphi$;
\item[(T3)] If $\Phi=t[P(\varphi_1),\ldots, P(\varphi_k)]$ is a compound modal formula, then $\Phi^\bullet=t[P(\varphi_1)^\bullet,\ldots, P(\varphi_k)^\bullet]=t[p_{\varphi_1},\ldots, p_{\varphi_k}]$. 
\end{itemize}

\begin{remark}[Probabilistic substitution]\label{remark:probsub}
The translation just introduced between modal formulas from $\pfm$ and propositional \luk\ formulas on variables $p_{\varphi_i}$'s might suggest to define a notion of substitution on modal formulas from $\pfm$ as any map 
\begin{equation}\label{eqSubsti}
\sigma:\{P(\varphi_1),\ldots, P(\varphi_k)\}\to \pfm.
\end{equation}
Any such $\sigma$ gives, modulo $^\bullet$, a typical \luk\ substitution. However, notice that $\vdash_{FP}$ does not satisfy the property of {\em substitution invariance}\footnote{Recall that  the entailment relation $\vdash_{\mathcal{L}}$ of a sentential logic $\mathcal{L}$ satisfies the property of substitution invariance (or {\em structurality}) if for every set of formulas $\Gamma\cup\{\varphi\}$ and for every substitutions $\sigma$, if $\Gamma\vdash_{\mathcal{L}}\varphi$, then $\sigma\Gamma\vdash_{\mathcal{L}}\sigma\varphi$ (see \cite[Definition 1.4]{FontBook} for further details).} under maps defined as in (\ref{eqSubsti}). Indeed, consider $\Phi=(P(x\vee \neg x))^2$ and $\Psi=(\neg P(x\vee \neg x))^2\vee (P(x\vee \neg x))^2$. Notice that $\mathscr{C}_{x\vee \neg x}=[1/2, 1]$ and, for all $\beta\in [1/2, 1]$, 
$$
\beta\models (P(x\vee \neg x))^2\leftrightarrow(\neg P(x\vee \neg x))^2\vee (P(x\vee \neg x))^2.
$$
Indeed, $\beta^2= (1-\beta)^2\vee\beta^2$ since for all $\beta\geq 1/2$, $(1-\beta)^2\leq\beta^2$. Therefore, by Corollary \ref{ComplCoherBooks}, 
$$
\vdash_{FP}\Phi\leftrightarrow\Psi.
$$
Consider the map $\sigma: \{P(x\vee\neg x)\}\to \pfm$ such that $\sigma(P(x\vee\neg x))=P(y)$. Then, $\not\vdash_{FP}\sigma(\Phi)\leftrightarrow\sigma(\Psi)$, where of course $\sigma(\Phi)=(P(y))^2$ and $\sigma(\Psi)=(\neg P(y))^2\vee P(y)^2$. Indeed, notice that  $\mathscr{C}_{y}=[0, 1]$. Thus, if we pick any $\beta \in [0, 1/2)$, we get that  $\beta^2 = 0$ and $ (1-\beta)^2\vee\beta^2 > 0$, since $1 - \beta > 1/2$. Therefore, again by Corollary \ref{ComplCoherBooks}, 
$$
\not\vdash_{FP}\sigma(\Phi)\leftrightarrow\sigma(\Psi).
$$
\end{remark}
The rationale behind the previous remark is that, since we want to regard atomic modal formulas of $\FPL$ as variables, in order for $\vdash_{FP}$ to satisfy substitution invariance, we need to make sure that such ``variables'' are evaluated in coherent sets. This fact leads to the following notion.
\begin{definition}[Probabilistic substitution]\label{defProbSub}
Let $\mathcal{E}=\{\varphi_1,\ldots, \varphi_k\}$ a set of events. A map $\sigma:\{P(\varphi_i)\mid i=1,\ldots, k\}\to \pfm$ is a {\em probabilistic substitution} if for all $\Phi,\Psi\in \pfm$ on atomic modal formulas $\{P(\varphi_i)\mid i=1,\ldots, k\}$, if $\vdash_{FP}\Phi\leftrightarrow\Psi$, then $\vdash_{FP}\sigma(\Phi)\leftrightarrow\sigma(\Psi)$. 
\end{definition}

The translation $^\bullet$ allows to translate deductions of $\FPL$ to \luk\ logic, by means of an infinite theory that interpretes all the instances of probability axioms. 
 Since standard completeness of \luk\ logic does not extend to deductions from infinite theories, the interaction between the syntax and the standard semantics of \luk\ logic fails when dealing with infinite theories.
A way to solve this issue is to locally reduces the deducibility relation of $\FPL$ to \luk\ calculus via de Finetti's coherence. The following theorem shows how to characterize deductions of $\FPL$  syntactically in \luk\ logic, from a geometrical viewpoint via coherent sets, and by an algebraic point of view in MV-algebras. For the next statement recall how $\chi_\mathcal{E}$ is defined in Remark \ref{rem:Algo}. 
\begin{theorem}\label{thm:red1}
Let $\Phi,\Psi\in \pfm$ and let $\mathcal{E}$ be the set of events occurring in them. Then, the following conditions are equivalent:
\begin{enumerate}
\item $\Phi\vdash_{FP}\Psi$;
\item $\chi_{\mathcal{E}},\Phi^\bullet\vdash_{\lu}\Psi^\bullet$;
\item $\mathscr{C}_{\mathcal{E}}\cap \mo(\Phi^\bullet)\subseteq\mo(\Psi^\bullet)$;
\item The quasiequation $((\chi_\mathcal{E}{\rm\; and\; }\Phi^\bullet) {\rm\; implies\; }\Psi^\bullet)$ holds in all MV-algebras. 
\end{enumerate}
\end{theorem}
\begin{proof}
(1)$\Rightarrow$(2) Let $\Phi=t[P(\varphi_1),\ldots, P(\varphi_k)]$ and $\Psi=r[P(\psi_1),\ldots, P(\psi_m)]$ be modal formulas in $\pfm$ and
assume that $\Phi\vdash_{FP}\Psi$. Moreover, let $e$ be a $[0,1]_{MV}$-model of $\chi_{\mathcal{E}}$ and $\Phi^\bullet$. Since $e(\chi_{\mathcal{E}})=1$, and by definition of the translation map $^\bullet:\pfm \to {\bf Fm}$,  it follows that  
$$
\langle e(p_{\varphi_1}),\ldots, e(p_{\varphi_k}), e(p_{\psi_1}),\ldots, e(p_{\psi_m})\rangle\in\oneset(\chi_{\mathcal{E}})= \mathscr{C}_{\mathcal{E}}.
$$
In other words, the assignment $\beta:\tau\mapsto e(p_{\tau})$ for $\tau\in \{\varphi_i, \psi_j\mid i=1,\ldots, k, j=1,\ldots,m\}$ is coherent. Moreover, the same $[0,1]_{MV}$-valuation $e$ is a model of $\Phi^\bullet$ and hence, 
$$
t^{[0,1]_{MV}}[e(p_{\varphi_1}),\ldots, e(p_{\varphi_k})]=t^{[0,1]_{MV}}[\beta(\varphi_1),\ldots, \beta(\varphi_k)]=1.
$$ 
By hypothesis, $\Phi\vdash_{FP}\Psi$ and $\beta$ is a model of $\Phi$. Therefore, by Corollary \ref{ComplCoherBooks}, $\beta$ is a model of $\Psi$ as well. That is to say, 
$$
r^{[0,1]_{MV}}[\beta(\psi_1),\ldots, \beta(\psi_m)]=r^{[0,1]_{MV}}[e(p_{\psi_1}),\ldots, e(p_{\psi_m})]=1
$$
showing that $e$ is a model of $\Psi^\bullet$ as required.
\vspace{.1cm}

(2)$\Rightarrow$(1) The argument is similar to the previous one. Indeed, every $[0,1]_{MV}$-model $e$ of $\chi_\mathcal{E}$ is any coherent book $\beta$ satisfying $\Phi$. Thus, by hypothesis $e$ models $\Psi^\bullet$ and thus $\beta\models \Psi$ as well, and the claim follows from Corollary \ref{ComplCoherBooks}.  
\vspace{.1cm}

(2)$\Leftrightarrow$(3) directly follows from Proposition \ref{PropPolyForm} (3) while (2)$\Leftrightarrow$(4) is an immediate consequence of the fact that MV-algebras are the equivalent algebraic semantics of \luk\ logic as we recalled in Subsection \ref{sec21}.
\end{proof}
Observe that in the statement of the above result we  assumed the  formulas $\Phi$ and $\Psi$ to be on the same set of events, without loss of generality due to Proposition \ref{propNumVariables}. 

The previous theorem should have clarified the reason why we spoke, at the beginning of this section, of \emph{local} reduction. Indeed, as we will further develop in the next section, the provability of a modal formula $\Phi$ in $\FPL$ is encoded by the deducibility in \luk\ logic of the translated formula $\Phi^\bullet$ from another propositional formula that indeed depends on $\Phi$ itself. In this precise sense the encoding of probabilistic to propositional entailment is local.

We end this section with some consequences of the previous theorem. \begin{corollary}
The deducibility relation of $\FPL$ is decidable.
\end{corollary}
\begin{proof}
By Theorem \ref{thm:red1}, each deduction $\Phi\vdash_{FP}\Psi$ holds in  $\FPL$  iff the corresponding translated deduction $\chi_{\mathcal{E}}, \Phi^\bullet\vdash \Psi^\bullet$ holds in \luk\ logic. $\Phi^\bullet$ and $\Psi^\bullet$ are obtained algorithmically from $\Phi$ and $\Psi$; moreover, $\chi_{\mathcal{E}}$ is computed as in Remark   \ref{rem:Algo}.  
The claim then follows from the fact that the deducibility relation of \luk\ logic is decidable \cite{W73}.
\end{proof}

\begin{corollary}
$\FPL$ has a local deduction theorem: for all formulas $\Phi,\Psi\in \pfm$, $\Phi\vdash_{FP}\Psi$ iff there exists $n \in \mathbb{N}$ such that $\vdash_{FP}\Phi^{n} \to\Psi$.
\end{corollary}
\begin{proof}
By Theorem \ref{thm:red1}, $\Phi\vdash_{FP}\Psi$ iff, given $\cc E$ the set of events in $\Phi$ and $\Psi$, it holds $\chi_{\mathcal{E}},\Phi^\bullet\vdash_{\lu}\Psi^\bullet$, iff $\chi_{\mathcal{E}} \vdash_{\lu} (\Phi^\bullet)^{n} \to\Psi^\bullet$ for some $n \in \mathbb{N}$, since \luk\ logic has a local deduction theorem (\cite[Theorem 1.7]{MuAdvanced}). Since $(\Phi^\bullet)^{n} \to\Psi^\bullet = (\Phi^{n} \to \Psi)^{\bullet}$, this happens iff $\chi_{\mathcal{E}} \vdash_{\lu} (\Phi^{n} \to \Psi)^{\bullet}$, and applying Theorem \ref{thm:red1} again, this is equivalent to $\vdash_{FP}\Phi^{n} \to\Psi$ and the proof is completed.
\end{proof}

\section{A local algebraic semantics for \FPL}\label{sec5}
In the sense of Theorem \ref{thm:red1}, MV-algebras constitute a semantics for $\FPL$. Indeed, checking validity of theorems and deductions of $\FPL$ corresponds to checking the validity of quasiequations in the variety of all MV-algebras. In this section we will show that we can actually restrict to a special class of {\em projective} MV-algebras.

\subsection{Coherent MV-algebras and projectivity}\label{sec51}

The class of MV-algebras, called {\em coherent MV-algebras}, that we define later in this section is meant to capture coherent books on events via a suitable quotient of a free MV-algebra. More precisely, if $\mathcal{E}$ is a set of events, say $\{\varphi_1,\ldots, \varphi_k\}$, and $\mathscr{C}_\mathcal{E}$ denotes as usual the set of all coherent assignments on $\mathcal{E}$, then $\free(\mathcal{E})/\mathscr{C}_\mathcal{E}$ is a prototypical example of a coherent MV-algebra. 
Thus, by Theorem \ref{thmCoheGeo}, for every set of events  $\mathcal{E}$, $\free(\mathcal{E})/\mathscr{C}_\mathcal{E}$ is finitely presented.

As the following result shows, $\free(\mathcal{E})/\mathscr{C}_\mathcal{E}$ encodes the probabilistically coherent books on events $\varphi_1,\ldots,\varphi_k$.
\begin{proposition}\label{propoHomoCohe}
For every finite set of events $\mathcal{E}=\{\varphi_1,\ldots,\varphi_k\}$, there exists a one-one correspondence between homomorphisms of  $\free(\mathcal{E})/\mathscr{C}_\mathcal{E}$ to $[0,1]_{MV}$ and coherent books on $\mathcal{E}$.
\end{proposition}
\begin{proof}
The proof is based on the general fact that, for every finitely presented MV-algebra ${\bf A}\cong\free(k)/\mathscr{P}$, the set of homomorphisms of ${\bf A}$ to $[0,1]_{MV}$ is in one-one relation with the points of $\mathscr{P}$, see \cite[Corollary 6.4]{MuAdvanced}. Let us call $\lambda$ the bijection between homomorphisms of $\free(\mathcal{E})/\mathscr{C}_\mathcal{E}$ to $[0,1]_{MV}$ and points of $\mathscr{C}_\mathcal{E}$. Thus, for every homomorphism $h:\free(\mathcal{E})/\mathscr{C}_\mathcal{E}\to[0,1]_{MV}$, let $\lambda(h)={\bf x}_h=\langle x_1,\ldots, x_k\rangle$. By the very definition of $\mathscr{C}_\mathcal{E}$ and Theorem \ref{thmCoheGeo}, the book $\beta:\varphi_i\mapsto x_i$ is coherent and this map associating points of $\mathscr{C}_\mathcal{E}$ to coherent books on $\mathcal{E}$ is clearly a bijection.  
\end{proof}

We previously observed that for every set $\mathcal{E}$ of events, $\free(\mathcal{E})/\mathscr{C}_\mathcal{E}$ is finitely presented. Our next result  shows that every such algebra is actually {\em projective} in the variety $\mathsf{MV}$ of MV-algebras. 

Before proving it, recall that an algebra ${\bf A}$ is {\em projective} in a class $\mathsf{K}$ of algebras in the same signature if for any $\alg B, \alg C \in \mathsf{K}$ and homomorphisms $f: \alg A \to \alg B$ and $g: \alg C \to \alg B$, with $g$ surjective, there exists an homomorphism $h: \alg A \to \alg C$ such that $g \circ h = f$.
If $\alg A$ is a finitely generated algebra and $\vv V$ is a variety, equivalently, $\alg A$ is projective iff it is a {\em retract} of a finitely generated free algebra. That is to say, if and only if 
for a free MV-algebra $\free(n)$, there are homomorphisms $i: {\bf A}\to \free(n)$ and $j:\free(n)\to{\bf A}$ 
such that $j\circ i$ is the identity homomorphism $id_{\bf A}$ of ${\bf A}$. Clearly, $i$ is an embedding, while $j$ is surjective. 

The following result from \cite{CabMu1} characterizes and shows properties of projective MV-algebras. For that, recall the following:
\vspace{.1cm}

(a) A map $\eta:[0,1]^k\to[0,1]^k$ is called a {\em $\mathbb{Z}$-retraction} if $\eta\circ\eta=\eta$ and $\eta$ is continuous, piecewise (affine) linear, and each of its pieces has integer coefficients. As observed in \cite[Lemma 2.4]{MS12} a map $\eta$ as above is a $\mathbb{Z}$-retraction iff there exist McNaughton functions $f_1,\ldots, f_k:[0,1]^k\to[0,1]$ such that  $\eta=\langle f_1,\ldots, f_k\rangle$, that is to say, for every $x\in[0,1]^k$, $\eta(x)=\langle f_1(x),\ldots, f_k(x)\rangle$. If there exists a $\mathbb{Z}$-retraction $\eta$ of $[0,1]^k$ onto $\mathscr{P}$, we say that $\mathscr{P}$ is a {\em $\mathbb{Z}$-retract} of $[0,1]^k$.
\vspace{.1cm}

(b) A map $\tau$ between two rational polyhedra $\mathscr{P}$ and $\mathscr{Q}$ of $[0,1]^k$ is a {\em $\mathbb{Z}$-homeomorphism} if $\tau$ is a homeomorphisms and there exist McNaughton functions $f_1,\ldots, f_k$ such that $\tau=\langle f_1,\ldots, f_k\rangle$. For a later use, observe that $\mathbb{Z}$-homeomorphisms may exist between a convex polyhedron $\mathscr{P}$ and a non-convex polyhedron $\mathscr{Q}$. That is to say, convexity is not preserved by $\mathbb{Z}$-homeomorphisms.

\vspace{.1cm}

(c) A set $X\subseteq[0,1]^k$ is said to be {\em star-shaped} if there exists an element $p\in X$ (called a {\em pole} of $X$) such that, for every $y\in X$, the linear segment $[p, y]$ is contained in $X$, see \cite{Holmes}.

\begin{theorem}[{\cite[Theorems 1.2, 1.4]{CabMu1}}]\label{thm:CabMund}
(1) A finitely generated MV-algebra ${\bf A}$ is projective iff ${\bf A}$ is isomorphic to $\free(n)/\mathscr{P}$ for some $\mathbb{Z}$-retract $\mathscr{P}$ of $[0,1]^n$.

(2) If $\mathscr{P}\subseteq[0,1]^n$ is a star-shaped rational polyhedron  with a pole $p\in \{0,1\}^n$, then $\free(n)/\mathscr{P}$ is projective. 
\end{theorem}

The next result provides a characterization of coherent sets through projective MV-algebras.

\begin{theorem}\label{propZRet}
Let 
$\mathscr{C}\subseteq[0,1]^{k}$ be convex. Then the following conditions are equivalent:
\begin{enumerate}
\item $\mathscr{C}$ is a coherent set, i.e.,  $\mathscr{C}=\mathscr{C}_\mathcal{E}$ for some set $\mathcal{E}=\{\varphi_1,\ldots, \varphi_k\}$ of events;
\item $\mathscr{C}$ is a $\mathbb{Z}$-retract of $[0,1]^{k}$;
\item $\free(k)/\mathscr{C}$ is  a projective MV-algebra.
\end{enumerate}
\end{theorem}

\begin{proof}
(1)$\Rightarrow$(3) Since $\mathscr{C}$ is convex, it is star-shaped and every point is a pole. Moreover,
if $\mathscr{C}$ is a coherent set, there exists a set of events  $\mathcal{E}=\{\varphi_1,\ldots, \varphi_k\}$ such that $\mathscr{C}=\mathscr{C}_\mathcal{E}$. By Theorem \ref{thmCoheGeo}, $\mathscr{C}$ contains a Boolean point which clearly is a pole. Therefore  $\free(k)/\mathscr{C}$ is projective from Theorem \ref{thm:CabMund} (2). 
\vspace{.1cm}

(3)$\Rightarrow$(2) Assume that $\free(k)/\mathscr{C}$ is projective. By Theorem \ref{thm:CabMund} (1), $\free(k)/\mathscr{C}$ is isomorphic to $\free(k)/\mathscr{P}$ for some rational polyhedron $\mathscr{P}$ which is a $\mathbb{Z}$-retract of $[0,1]^k$. By \cite[Corollary 3.10]{MuAdvanced}, it then follows that $\mathscr{C}$  and $\mathscr{P}$ are $\mathbb{Z}$-homeomorphic and  \cite[Lemma 17.6]{MuAdvanced} proves that $\mathbb{Z}$-homeomorphisms preserve $\mathbb{Z}$-retracts. Thus, $\mathscr{C}$ is a $\mathbb{Z}$-retract of $[0,1]^{k}$. 
\vspace{.1cm}

(2)$\Rightarrow$(1) Assume that $\mathscr{C}$ is a $\mathbb{Z}$-retract of $[0,1]^{k}$ 
by a $\mathbb{Z}$-retraction  $\eta:[0,1]^{k}\to\mathscr{C}$. Then, there are McNaughton functions $f_1,\ldots, f_k:[0,1]^k\to[0,1]$ such that, for all $x=\langle x_1,\ldots, x_k\rangle\in [0,1]^k$,
$
\eta(x)=\langle f_1(x),\ldots,f_k(x)  \rangle\in \mathscr{C}
$
and $\eta\circ\eta=\eta$. Indeed, 
\begin{equation}\label{eqZRet}
\mathscr{C}=\{\langle f_{1}(x),\ldots, f_{k}(x)\rangle\mid x\in [0,1]^k\}.
\end{equation}
Let $\mathcal{E}=\{\varphi_1,\ldots,\varphi_k\}$ such that for all $i=1,\ldots, k$, $f_i=f_{\varphi_i}$. Then, by definition of coherent set and (\ref{eqZRet}), one has that 
$$
\mathscr{C}_\mathcal{E}=\overline{\rm co}\{\langle f_{\varphi_1}(x),\ldots, f_{\varphi_k}(x)\rangle\mid x\in [0,1]^k\}=\overline{\rm co}(\mathscr{C})=\mathscr{C},
$$
where the last equality holds because $\mathscr{C}$ is convex by hypothesis. 
\end{proof}

Now, let $k$ be any positive integer and let $\mathscr{C}$ be a coherent subset of $[0,1]^k$, that is to say, let $\mathcal{E}=\{\varphi_1,\ldots,\varphi_k\}$ be a set of events such that $\mathscr{C}=\mathscr{C}_\mathcal{E}$. Then, it is easy to see that the map $\lambda: \free(\mathcal{E})/\mathscr{C}_\mathcal{E}\to \free(k)/\mathscr{C}$ sending, for every $i=1,\ldots, k$, the generator $[\varphi_i]_{\mathscr{C}_\mathcal{E}}$ of $\free(\mathcal{E})/\mathscr{C}_\mathcal{E}$ to the generator $[x_i]_\mathscr{C}$ of $\free(k)/\mathscr{C}$ determines a isomorphism between the two MV-algebras. Thus, we define coherent MV-algebras as follows.

\begin{definition}\label{defCoheMV}
An MV-algebra is said to be {\em coherent} if it is isomorphic to  $\free(k)/\mathscr{C}$ where $\mathscr{C}$ is a coherent subset of $[0,1]^k$, for some $k \in \mathbb{N}$. We denote this class of algebras by $\mathsf{coMV}$.
\end{definition}
Notice that all finitely generated free MV-algebras are coherent by Theorem  \ref{propZRet}.
 Indeed, $\free(k)\cong\free(k)/[0,1]^k$ and $\free(k)$ is projective.
Moreover notice that $[0,1]^k$ is a coherent set and indeed $[0,1]^k=\mathscr{C}_{\mathcal{E}}$ for $\mathcal{E}$ being the set of propositional variables $x_1,\ldots, x_k$.

By definition, the class $\CMV$ is closed under isomorphisms. Thus, an algebra ${\bf A} \in \CMV$ might be of the form $\free(k)/\mathscr{C}$ where $\mathscr{C}$ is not necessarily convex. Indeed, a direct consequence of the duality put forward in \cite{MS13} is that two finitely presented MV-algebras are isomorphic if and only if their respective polyhedra are $\mathbb{Z}$-homeomorphic  and, as we recalled in the above point (b), convexity is not preserved under $\mathbb{Z}$-homeomorphism. However, isomorphisms preserves coherent assignments.
\begin{remark}\label{remHomoCohe}
By Definition \ref{defCoheMV}, for every ${\bf A}\in \CMV$ there exists a (not necessarily unique) set of events $\mathcal{E}$ such that ${\bf A}$ is isomorphic to $\free(\mathcal{E})/\mathscr{C}_{\mathcal{E}}$. Every homomorphism $h: {\bf A}\to [0,1]_{MV}$ determines a homomorphism $h':\free(\mathcal{E})/\mathscr{C}_\mathcal{E}\to[0,1]_{MV}$  composing $h$ and the  isomorphism $\lambda: \free(\mathcal{E})/\mathscr{C}_\mathcal{E}\to {\bf A}$. Therefore,  by Proposition \ref{propoHomoCohe} the book $\beta:\varphi_i\mapsto h(\lambda([\varphi_i)])$ is coherent. 
\end{remark}

Notice also that coherent MV-algebras are not closed under the universal algebraic operators of homomorphic images, subalgebras and direct products (indeed not even projective MV-algebras are), thus they neither are a variety nor a quasivariety. 

The following is a direct consequence of Definition \ref{defCoheMV}, Theorem \ref{propZRet} and the fact that projective algebras are closed under isomorphic images.

\begin{corollary}\label{corProjProp}
Every coherent MV-algebra is projective in the variety of MV-algebras, and thus also in the class $\CMV$. In particular, for every ${\bf A}\in \CMV$ there exists a free MV-algebra $\free(k)$ and homomorphisms $i: {\bf A}\to\free(k)$ and $j:\free(k)\to{\bf A}$ such that $j\circ i=id_{\bf A}$.
\end{corollary}

A direct inspection on  the proof of Theorem \ref{propZRet} shows a further property of coherent MV-algebras and coherent sets. Indeed, consider a set of events $\mathcal{E}=\{\varphi_1,\ldots, \varphi_k\}$ where the $\varphi_i$'s are written in, say, $n$ propositional variables. Then, the equivalence between (1) and (2) in Theorem \ref{propZRet} tells us that $\mathscr{C}_\mathcal{E}=\eta([0,1]^\mathcal{E})$ where $\eta$ is a $\mathbb{Z}$-retraction. Thus, there are McNaughton functions $f_1,\ldots, f_k:[0,1]^k\to[0,1]$ such that $\eta=\langle f_1,\ldots, f_k\rangle$. Call $\varphi_1',\ldots, \varphi_k'$ the formulas such that $f_i=f_{\varphi'_i}$ and let $\mathcal{E}'=\{\varphi_1',\ldots,\varphi_k'\}$. Notice that each $\varphi'_i$ is written in $k$ variables and $\mathscr{C}_\mathcal{E}=\mathscr{C}_{\mathcal{E}'}$. In other words, 
$$
\mathscr{C}_\mathcal{E}=\overline{\rm co}\{\langle f_{\varphi_1}(x),\ldots, f_{\varphi_k}(x)\rangle\mid x\in [0,1]^n\}=\{\langle f_{\varphi_1'}(\by),\ldots, f_{\varphi_k'}(\by)\rangle\mid \by\in [0,1]^k\}.
$$
Therefore, the following holds.
\begin{corollary}\label{corNtoK}
For every set of events $\mathcal{E}=\{\varphi_1,\ldots,\varphi_k\}$ in $n$ variables, there exists a set of events $\mathcal{E}'=\{\varphi'_1,\ldots,\varphi'_k\}$ in $k$ variables such that $\{\langle f_{\varphi_1'}(\by),\ldots, f_{\varphi_k'}(\by)\rangle\mid \by\in [0,1]^k\}$ is convex and it coincides with $\mathscr{C}_{\mathcal{E}}$. 
\end{corollary}

Now, we turn our attention to a further property that coherent MV-algebras enjoy and that will be used in the next section. To this end, recall that a class of algebras $\mathsf{K}$ in the same signature has the {\em joint embedding property} if for all ${\bf A}, {\bf B}\in \mathsf{K}$, there exists a ${\bf C}\in \mathsf{K}$ and embeddings $h_{\bf A}: {\bf A}\to {\bf C}$ and $h_{\bf B}:{\bf B}\to{\bf C}$. 
\begin{proposition}\label{JEP}
The class $\CMV$ of coherent MV-algebras has the joint embedding property.
\end{proposition}
\begin{proof}
Let ${\bf A}$ and ${\bf B}$ two coherent MV-algebras that without loss of generality, we will think, respectively, as $\free(\mathcal{E}_1)/\mathscr{C}_{\mathcal{E}_1}$ and $\free(\mathcal{E}_2)/\mathscr{C}_{\mathcal{E}_2}$. Let us call $\mathcal{E}_3=\mathcal{E}_1\cup\mathcal{E}_2$. Then, $\mathscr{C}_{\mathcal{E}_3}$ is a rational polytope of $[0,1]^{\mathcal{E}_3}=[0,1]^{\mathcal{E}_1\cup \mathcal{E}_2}$. Consider the projection maps $\pi_1:[0,1]^{\mathcal{E}_3}\to[0,1]^{\mathcal{E}_1}$ and $\pi_2:[0,1]^{\mathcal{E}_3}\to[0,1]^{\mathcal{E}_2}$. By Proposition \ref{prop:restrictC}, one  has that $\mathscr{C}_{\mathcal{E}_1}=\pi_1[\mathscr{C}_{\mathcal{E}_3}]$ and $\mathscr{C}_{\mathcal{E}_2}=\pi_2[\mathscr{C}_{\mathcal{E}_3}]$. 

Notice that, for $i=1,2$, each projection $\pi_{i}$ is a $\mathbb{Z}$-map. Thus, its  associated dual map $h_i: \free(\mathcal{E}_i)/\mathscr{C}_{\mathcal{E}_i}\to \free(\mathcal{E}_3)/\mathscr{C}_{\mathcal{E}_3}$ is a homomorphism by \cite[Lemma 3.3]{MS13}. Furthermore, $h_i$ is injective by the duality theorem of \cite{MS13} (by direct inspection on how the functor named $\mathscr{M}$ in \cite[\S3.2]{MS13}  acts on $\mathbb{Z}$-maps).
Thus, each $h_i$ is an embedding of $\free(\mathcal{E}_i)/\mathscr{C}_{\mathcal{E}_i}$ into $\free(\mathcal{E}_3)/\mathscr{C}_{\mathcal{E}_3}$. Moreover, the two maps are such that, for all $\alpha\in\mathcal{E}_i$, 
$
h_i([\alpha]_{\mathscr{C}_{\mathcal{E}_i}})=[\alpha]_{\mathscr{C}_{\mathcal{E}_3}}.
$ 
\end{proof}

\subsection{Validity in $\CMV$ and provability in \FPL}\label{sec62}

We are now going to define the notion of semantical consequence that we mean to use for $\CMV$-algebras. The idea is that, by Theorem \ref{ComplCoherBooks}, theorems and deductions in $\FPL$ ground on coherent assignments. In turn, the latter corresponds to  homomorphisms of a coherent MV-algebra to $[0,1]_{MV}$ by Proposition \ref{propoHomoCohe} and Remark \ref{remHomoCohe}.

Now recall that given an MV-term $t(x_{1}\ldots x_{k})$, and $\alg A$ an MV-algebra, by 
$
\models_{\alg A} t(x_{1},\ldots, x_{k})
$ 
one usually means that  any assignment $h$ of the variables in $X = \{x_{1},\ldots, x_{k}\}$ to elements of $\alg A$: $h(x_{1}) = a_{1}\ldots h(x_{k}) = a_k$, (uniquely) extends to a homomorphism $h$ from the term algebra $\term(X)$ to $\alg A$ (see \cite[\S 1.1]{GJKO} for details) such that, 
$
t^{\alg A}(a_{1}, \ldots, a_{k}) = \top_{\alg A}.
$
Since MV-algebras form a variety, one can equivalently say that $\models_{\alg A} t(x_{1}\ldots x_{k})$ if for all homomorphisms $h: \free(k) \to \alg A$ (where $\free(k)$ is generated by $x_{1} \ldots x_{k}$), $h(x_{i}) = a_{i}$ for $a_{1} \ldots a_{n}$, $t^{\alg A}(a_{1}, \ldots a_{k}) = \top_{\alg A}.$

For the case of $\FPL$, however, formulas have two layers: an outer MV-term, and inner MV-terms (recall Subsection \ref{sec23}). In order to take care of this, we shall define a notion of \emph{coherent valuation}, based on a coherent MV-algebra that acts as a ``local'' version of the free algebra.

\begin{notation}\label{not:Events}
	From now on, in order to avoid any possible confusion, given a set of events $\cc E = \{\varphi_1, \ldots, \varphi_k\}$ we will see the free MV-algebra $\free(\cc E)$ as generated by the variables $p_{\varphi_1}, \ldots, p_{\varphi_k}$ (instead of  using the events $\varphi_1, \ldots, \varphi_k$ to denote the variables). We shall also write $\term(\cc E)$ to denote the term algebra in the \luk\ language over the variables $p_{\varphi_1}, \ldots, p_{\varphi_k}$. 
\end{notation}

Let now $\Phi=t[P(\varphi_1),\ldots, P(\varphi_k)]$ be a formula in $\FPL$, and call $\mathcal{E}=\{\varphi_1,\ldots,\varphi_k\}$. Let $\Phi^\bullet=t(p_{\varphi_1},\ldots, p_{\varphi_k})$ be the translation of $\Phi$ to \luk\ language as in Section \ref{sec4}, and let $\mu_{\cc E}$ be the natural epimorphism $\mu_\mathcal{E}: \term(\cc E) \to \free(\cc E)$ (where we use epimorphism in the universal algebraic sense, that is, to mean a surjective homomorphism). Consider the valuation  $v_{\mathcal{E}}$ of $\Phi^\bullet$ to $\free(\cc{E})/\scr C_{\cc E}$ to be the homomorphism from the term algebra $\term(\cc E)$ to $\free(\mathcal{E})/\mathscr{C}_\mathcal{E}$ that extends the map 
$$
v_{\mathcal{E}}(p_{\varphi_i}) = [\mu_{\cc E}(p_{\varphi_{i}})]_{\mathscr{C}_\mathcal{E}}, \; \mbox{ for } i = 1, \ldots, k.
$$

The above is hence a fixed interpretation for each formula $\Phi\in \pfm$, and it evaluates $\Phi$  in the  coherent algebra $\free(\mathcal{E})/\mathscr{C}_\mathcal{E}$. 
We can extend such an evaluation to coherent MV-algebras isomorphic to $\free(\mathcal{E})/\mathscr{C}_\mathcal{E}$ in the following way.
\begin{definition}
	Consider a set of events $\cc E$, and let $\alg A \in \CMV$ isomorphic to $\free(\mathcal{E})/\mathscr{C}_\mathcal{E}$ via a map $f: \free(\mathcal{E})/\mathscr{C}_\mathcal{E} \to \alg A$. The composition $f \circ v_{\cc E}: \term(\cc E) \to \alg A$ is said to be a {\em coherent valuation} of the formulas in $\pfm$ over the set of events $\cc E$. 
\end{definition}

In accordance to what we showed in Proposition \ref{propoHomoCohe}, we introduce the following notion of semantic derivability that uses homomorphisms of coherent MV-algebras in $[0,1]_{MV}$.
\begin{definition}
Let $\Phi$ and $\Psi$ from $\pfm$ be over the set of events $\mathcal{E}$. 
We write $\Phi\models_{\mathcal{E}}\Psi$ if for every coherent MV-algebra $\alg A$ isomorphic to $\free(\mathcal{E})/\mathscr{C}_\mathcal{E}$, every coherent valuation $e: \term(\cc E) \to \alg A$, and every
homomorphism $h: \alg A \to[0,1]_{MV}$,  it holds that $h(e(\Phi^\bullet))=1$ implies  $h(e(\Psi^\bullet))=1$.
\end{definition}
 We are now in a position to show that coherent MV-algebras can be regarded as algebraic models for the probability logic $\FPL$.
\begin{theorem}\label{thm:coMVcomp}
Let $\Phi,\Psi\in \pfm$ and let $\mathcal{E}$ be the set of events occurring in them. Then, the following  are equivalent:
\begin{enumerate}
\item $\Phi\vdash_{FP}\Psi$;
\item $\Phi\models_{\mathcal{E}}\Psi$;
\end{enumerate} 
\end{theorem}
\begin{proof}
Let $\Phi,\Psi\in \pfm$ over events in $\mathcal{E}$. Notice that it is a straightforward consequence of the definition that $\Phi\models_{\mathcal{E}}\Psi$ if and only if for every homomorphism $h: \free(\mathcal{E})/\mathscr{C}_\mathcal{E} \to [0,1]_{MV}$, $h(v_{\cc E}(\Phi^\bullet)) = 1$ implies $h(v_{\cc E}(\Psi^\bullet)) = 1$.

Assume now that $\Phi\not\models_\mathcal{E} \Psi$. Hence, that is true if and only if there is a homomorphism $h: \free(\mathcal{E})/\mathscr{C}_\mathcal{E}\to[0,1]_{MV}$ such that $ h(v_{\cc E}(\Phi^\bullet))=1$  and  $h( v_{\cc E}(\Psi^\bullet))<1$. By Proposition \ref{propoHomoCohe}, homomorphism from $\free(\mathcal{E})/\mathscr{C}_\mathcal{E}$ to $[0,1]_{MV}$ are in one-one correspondence to coherent books on $\cc E$. In particular, the map $\beta_h:\varphi_i\in \mathcal{E}\mapsto h([\varphi_i])\in [0,1]$ is a coherent book on $\mathcal{E}$  (see also Remark \ref{remHomoCohe}). Then, $ h(v_{\cc E}(\Phi^\bullet))=1$  and  $h( v_{\cc E}(\Psi^\bullet))<1$ holds if and only if $\beta_h\models\Phi$, but $\beta_h\not\models \Psi$. The existence of a coherent book $\beta$ such that $\beta \models\Phi$, but $\beta \not\models \Psi$ is in turn equivalent to the fact that $\Phi\not\vdash_{FP}\Psi$ by Corollary \ref{ComplCoherBooks}. Thus the claim is settled.
\end{proof}

In light of the above proof, we notice that the algebras of the kind $\free(\cc E)/\scr{C}_{\cc E}$ are in some sense the \emph{standard models} of the deductions over the events in $\cc E$.

Let us end this section by  remarking that, although Theorem \ref{thm:coMVcomp} shows that the class of algebras $\CMV$ provides a semantics for $\FPL$, it is not its equivalent algebraic semantics neither in the sense of Lindenbaum-Tarski nor Blok-Pigozzi. Indeed, such a semantics would likely need to have two-sorted algebras as recently done in \cite{KM}.

\section{Probabilistic unification and its unification type}\label{sec6}
We now show how a probabilistic version of the unification problem for $\FPL$ has an algebraic equivalent in similar terms to the approach developed by Ghilardi  \cite{G97}. 

The usual way of formulating a symbolic unification problem for an algebraizable logic $\mathcal{L}$ is to consider a (finite) set of pairs of terms $\{(s_{i}, t_{i}) \mid i \in I\}$ over the language of $\mathcal{L}$, and to solve the unification problem means to find a substitution $\sigma$, called \emph{unifier}, of the variables occurring in the terms $s_{i}, t_{i}$ for $i \in I$ that makes the identities $\{\sigma(s_{i}) = \sigma(t_{i}) \mid i \in I\}$ valid in the equivalent algebraic semantics of $\mathcal{L}$. 
Unifiers can be ordered by generality in the following way: a substitution $\sigma_{1}$ is more general than a substitution $\sigma_{2}$ if there is another substitution $\tau$ such that $\tau \circ \sigma_{1} = \sigma_{2}$. This gives a preorder on the set of unifiers for a problem, thus we can consider the associated partial order (where the equivalence classes correspond to unifiers that are \emph{equally general}).
The \emph{unification type} of a problem is said to be: {\em unitary}, if the partial order of the unifiers has a maximum; {\em finitary}, if it does not have a maximum but it has finitely many maximal elements; {\em infinitary} if it instead has infinitely many maximal elements; {\em nullary} otherwise. The unification type of a logic $\mathcal{L}$ is the worst unification type occurring for a unification problem in $\mathcal{L}$.

Ghilardi shows that for an algebraizable logic with algebraic semantics $\vv V$, unification can also be studied by algebraic means. In particular, a unification problem corresponds to a finitely presented algebra $\alg A \in \vv V$, seen as generated by a finite set $X$ and presented by a finite set of identities $S$, so we write $\alg A = \alg F(X, S)$.
A solution (or unifier) is a homomorphism $u : {\bf A }\to {\bf P}$, where $\alg P$ is a projective algebra in $\vv V$. Algebraic unifiers can also be ordered by generality, by saying that a unifier $u_{1}: {\bf A }\to {\bf P}_{1}$ is more general than $u_{2}: {\bf A} \to {\bf P}_{2}$ if there is an homomorphism $p: {\bf P}_{1} \to {\bf P}_{2}$ such that $p \circ u_{1} = u_{2}$. This gives a preorder on algebraic unifiers, and thus considering the associated partial order, one can define a notion of algebraic unification type for the algebraic semantics of a logic. Ghilardi shows that symbolic and algebraic unification type coincide for algebraizable logics \cite{G97}.

 Marra and Spada applied  Ghilardi result in \cite{MS13} to the case of \luk\ logic and showed that the unification type of \luk\ logic is nullary. More precisely, they constructed a co-final chain of unifiers of order type $\omega$ for a specific MV-algebraic unification problem. Their proof uses the duality between finitely presented MV-algebras and  rational polyhedra, and indeed their argument is purely geometrical. 

 As we remarked at the end of Subsection \ref{sec62}, we are not in the presence of an equivalent algebraic semantics for $\FPL$ and hence Ghilardi's theorem does not apply to this case straightforwardly. However, an
analogous result for $\FPL$ can be proved. 

In particular, as we have seen, $\FPL$ is a logic that reasons about uncertainty measures of \luk\ events, and that essentially (in the sense of Theorem \ref{thm:red1}) treats atomic modal formulas as its variables. Therefore
we define a probabilistic version of a unification problem following this intution,
where MV-algebraic terms are unified considering probabilistic formulas as arguments.

\subsection{The symbolic and algebraic probabilistic unification problems}

We first introduce the main definitions for this section. Namely, the symbolic and the algebraic ways to regard a unification problem for the probability logic $\FPL$.

\begin{definition}\label{def:SPUI}
A {\em (symbolic) probabilistic unification problem} for $\FPL$ is a  set of $m$ identities 
\begin{equation}\label{eq:uniProb}
\mathcal{I}=\{t_i[P(\varphi_1),\ldots, P(\varphi_k)] = u_i[P(\varphi_1),\ldots, P(\varphi_k)]\mid i=1,\ldots, m\}.
\end{equation}
\end{definition}
Notice that the identities are assumed to be on the same set of events $\mathcal{E}=\{\varphi_1,\ldots, \varphi_k\}$, without loss of generality by Proposition \ref{propNumVariables}. 

Given a probabilistic unification problem $\mathcal{I}$ as above,   a {\em probabilistic unifier} for $\mathcal{I}$ is a probabilistic substitution as in Definition \ref{defProbSub} (and as motivated by Remark \ref{remark:probsub}), such that
  for all $i = 1, \ldots, m$:
\begin{equation}\label{eq:defprobunifier}
\vdash_{FP} t_{i}[\sigma(P(\varphi_{1})), \ldots, \sigma(P(\varphi_{k}))] \leftrightarrow u_{i}[ \sigma(P(\varphi_{1})), \ldots,\sigma(P(\varphi_{k}))]
\end{equation}
As in the propositional case,  also probabilistic unifiers can be ordered by generality,  in the following sense. Consider two probabilistic unifiers $\sigma, \tau$ for the above problem, say $\sigma(P(\varphi_{i})) =  r_{i} [P(\gamma_{1}),\ldots, P(\gamma_{t})]$ and $\tau(P(\varphi_{i})) =  r'_{i} [P(\gamma'_{1}),\ldots ,P(\gamma'_{l})]$ for $i = 1, \ldots, m$.  
Then we say that \emph{$\tau$ is more general that $\sigma$}, and write $\sigma \sqsubseteq \tau$, if there exists a probabilistic substitution  $\delta$ mapping each $P(\gamma'_{j})$ to a term $r''_{i} [P(\gamma_{1}),\ldots, P(\gamma_{t})]$ for $j = 1,\ldots, l$, such that $\sigma = \delta \circ \tau$ is provable in $\FPL$, that is to say
\begin{equation}\label{eq:composition}
\vdash_{FP} t[\sigma(P(\varphi_{1})), \ldots, \sigma(P(\varphi_{k}))] \leftrightarrow t[(\delta \circ \tau)(P(\varphi_{1})), \ldots,(\delta \circ \tau)(P(\varphi_{k}))].
\end{equation}
Notice that, in the expression above, the unifier $\sigma$ can be lifted out of the outer terms and hence it can be equivalently written as
$$\vdash_{FP} \sigma(t[\sigma(P(\varphi_{1})), \ldots, P(\varphi_{k})]) \leftrightarrow (\delta \circ \tau)(t[P(\varphi_{1}), \ldots,P(\varphi_{k})]).
$$
\begin{lemma}
Given a probabilistic unification problem $\cc I$, $\sqsubseteq$ is a preorder on its set of probabilistic unifiers.
\end{lemma}
\begin{proof}
The fact that $\sqsubseteq$ is reflexive follows from the fact that clearly the identity map is a probabilistic substitution over any set of atomic modal formulas. Let us then show that $\sqsubseteq$ is transitive. Suppose that $\sigma, \tau, \rho$ are probabilistic substitutions such that $\sigma \sqsubseteq \tau$ and $\tau \sqsubseteq \rho$. Then there exist $\delta, \delta'$ probabilistic substitutions defined over the appropriate sets of atomic modal formulas such that $\sigma = \delta \circ \tau$ and $\tau = \delta' \circ \rho$. Thus, consider $\delta'' = \delta \circ \delta'$. Since $\vdash_{FP}$ satisfies the property of substitution invariance with respect to $\delta$ and $\delta'$, this will also hold for their composition. Now, by $\tau = \delta' \circ \rho$ we mean that 
$
\vdash_{FP} \tau (t) \leftrightarrow \delta' \circ \rho(t)
$
for any term $t$ written over the appropriate set of atomic modal formulas. Then by the definition of a probabilistic substitution we derive that 
$
\vdash_{FP} \delta \circ \tau (t) \leftrightarrow \delta\circ (\delta' \circ \rho)(t)
$
and therefore 
$
\vdash_{FP} \delta \circ \tau (t) \leftrightarrow (\delta'' \circ \rho)(t).
$
Since by $\sigma = \delta \circ \tau$ we get 
$
\vdash_{FP} \sigma (t) \leftrightarrow \delta \circ \tau(t),
$
it follows, as desired, that 
$
\vdash_{FP} \sigma (t) \leftrightarrow \delta'' \circ \rho(t)
$
which means that $\sigma = \delta'' \circ \rho$ and then $\sqsubseteq$ is transitive.
\end{proof}
 Every preorder $\leq$ on a set $X$ induces a poset on the quotient defined by the equivalence relations $x\sim y$ iff $x\leq y$ and $y\leq x$. With an abuse of notation, but without danger of confusion, in what follows we will denote by $X$ the quotient $X/{\sim}$. 
\begin{definition}
For every probabilistic unification problem ${\cc I}$, we denote by $U_{\cc I}$ the set of unifiers for ${\cc I}$. By $\mathscr{S}_\mathcal{I}=(U_{\cc I}, \leq)$ we denote the poset induced by the preorder $\leq$ of equally general symbolic unifiers. 
\end{definition}

Clearly, the translation map $^\bullet$ from modal to propositional \luk\ formulas used in the previous sections allows to translate each probabilistic unification problem $\mathcal{I}$ as in (\ref{eq:uniProb}) on events $\mathcal{E}=\{\varphi_1,\ldots, \varphi_k\}$,  to the propositional \luk\ unification problem 
\begin{equation}\label{eqipallino}
\begin{array}{lll}
\mathcal{I}^\bullet&=&\{t_i[P(\varphi_1)^\bullet,\ldots, P(\varphi_k)^\bullet] = u_i[P(\varphi_1)^\bullet,\ldots, P(\varphi_k)^\bullet]\mid i=1,\ldots, m\}\\
&=&\{t_i[p_{\varphi_1},\ldots, p_{\varphi_k}] = u_i[p_{\varphi_1},\ldots, p_{\varphi_k}]\mid i=1,\ldots, m\}.
\end{array}
\end{equation}
Given a solution $\sigma$ for $\mathcal{I}$, its translation $\sigma^\bullet:p_{\varphi_i} \mapsto (\sigma(P(\varphi_i)))^\bullet$ is  such that, for all $i=1,\ldots, m$,
 
\begin{equation}\label{eq:uniftranslated}
\chi_{\sigma(\mathcal{E})}\vdash_{\lu} \sigma^\bullet(t_i[p_{\varphi_1},\ldots, p_{\varphi_k}])\leftrightarrow \sigma^\bullet(u_i[p_{\varphi_1},\ldots, p_{\varphi_k}]),
\end{equation}
where, with an abuse of notation, $\sigma(\mathcal{E})$ denotes, here and henceforth, the set of  events occurring in $\sigma(P(\varphi_1)), \ldots, \sigma(P(\varphi_k))$ and thus $\chi_{\sigma(\mathcal{E})}$ is the \luk\ formula as in Theorem \ref{thm:red1} corresponding to the coherent set of the events in $\sigma(\mathcal{E})$. 

Now we can notice that solving the probabilistic unification problem does not simply reduce to the usual unification in \luk\ logic, for the presence of the formula $\chi_{\sigma(\mathcal{E})}$ that takes care of the coherence of \luk\ valuations on the new events appearing after the substitution. However, we can rephrase the algebraic approach in this context, translating the problem to a finitely presented MV-algebra, and the solution to a homomorphism to a coherent MV-algebra, that we know to be projective by Corollary \ref{corProjProp}. 

It is now convenient to recall Notation \ref{not:Events}: given a set of events $\cc E = \{\varphi_1, \ldots, \varphi_k\}$, we  see the free MV-algebra $\free(\cc E)$ as generated by the variables $p_{\varphi_1}, \ldots, p_{\varphi_k}$.

\begin{definition}
An {\em algebraic probabilistic unification problem} is an MV-algebra determined by a set of events $\mathcal{E}= \{\varphi_1, \ldots, \varphi_k\}$ and a presentation (a finite set of identities) $\mathcal{P}$ over variables $p_{\varphi_1}, \ldots, p_{\varphi_k}$. In symbols ${\bf F}(\cc E, \cc P) = \free(\cc E)/(\scr{C}_{\cc E} \cap \scr{P}_{\cc P})$ where $\scr{P}_{\cc P}$ is the polyhedron associated to the congruence on $\free(\cc E)$ generated by $\cc P$. Given an algebraic probabilistic unification problem ${\bf F}(\cc E, \cc P)$, an {\em algebraic unifier for ${\bf F}(\cc E, \cc P)$} is a homomorphism $h: {\bf F}(\cc E, \cc P)\to {\bf C}$ where ${\bf C}$ is a coherent MV-algebra.
\end{definition}
Notice that each ${\bf F}(\cc E, \cc P)$ is a finitely presented MV-algebra, since both $\scr C_{\cc E}$ and $\scr P_{\cc P}$ are rational polyhedra, and thus so is their intersection. 

The definition we choose for an algebraic probabilistic unification problem, despite seeming ad hoc, is not restrictive: indeed, every finitely presented MV-algebra is an algebraic probabilistic unification problem. It suffices to take as $\cc E$ any set of \luk\ variables $X$. In this case,  for every presentation $\mathcal{P}$, we obtain the algebras ${\bf F}(X,\mathcal{P})$ that are algebraic unification problems as in the propositional case, since $\scr C_{X} = [0,1]^{n}$ where $n = |X|$. 

\begin{remark}
Notice that for some algebraic probabilistic unification problem ${\bf F}(\cc E, \cc P)$ it might happen that  $\scr{C}_{\cc E}$ and $\scr{P}_{\cc P}$ have void intersection. Since $\scr{C}_{\cc E}\neq \emptyset$ for all $\mathcal{E}$, this can happen either if $\scr{P}_{\cc P}=\emptyset$, meaning that the pairs appearing in the presentation ${\cc P}$ have no solution even in \luk\ logic, or if in fact  $\scr{P}_{\cc P}$ is not void but it does not intersect with $\scr{C}_{\cc E}$. The latter case intuitively means that, although the identities presented by ${\cc P}$ have solution  in \luk\ logic, the solutions are not coherent in the sense specified in the above sections. Also notice that, if $\scr{C}_{\cc E}\cap \scr{P}_{\cc P}=\emptyset$, then ${\bf F}(\cc E, \cc P)$ still is finitely presented and it coincides with the one-point, trivial, algebra. 
\end{remark}

Algebraic probabilistic unifiers can be ordered by generality by $\preceq$ as in the propositional case. Moreover, $\preceq$ is a preoder on the set $U_{{\bf F}(\cc E, \cc P) }$ of algebraic unifiers for an algebraic probabilistic unification problem ${\bf F}(\cc E, \cc P) $.

\begin{definition}\label{defAMP}
We denote by $\mathscr{A}_{{\bf F}(\cc E, \cc P) } =(U_{{\bf F}(\cc E, \cc P) }, \preceq)$ the poset of algebraic unifiers for an algebraic probabilistic unification problem  ${\bf F}(\cc E, \cc P) $ and  whose elements are equivalence classes of equally general unifiers.
\end{definition}

 \subsection{Ghilardi-like theorem for probabilistic unification}

Having a natural notion of both symbolic and algebraic unification for $\FPL$, we now  prove that the two approaches are equivalent. In order to do so, 
we will show that, given a symbolic probabilistic unification problem, we can find an algebraic problem with the same unification type, and viceversa. Let us first define  two maps that translate probabilistic unification problems and their unifiers to their algebraic counterpart and viceversa.

Take any symbolic probabilistic unification problem 
$$
\mathcal{I}=\{t_i[P(\varphi_1),\ldots, P(\varphi_k)] = u_i[P(\varphi_1),\ldots, P(\varphi_k)]\mid i=1,\ldots, m\},
$$
and let $\mathcal{E} =\{\varphi_1, \ldots, \varphi_k\}$, $\mathscr{C}_{\mathcal{E}}$ be the coherent set of $\varphi_1,\ldots, \varphi_k$, and $\mathscr{P}_{\mathcal{I}^\bullet}$ be the polyhedron determined by the set of equations in (\ref{eqipallino}). Finally, let 
\begin{equation}\label{eq:defAI}
{\sf A}(\mathcal{I}) = {\bf F}(\cc E, \cc I^{\bullet}) = \free(\mathcal{E})/(\mathscr{C}_{\mathcal{E}}\cap\mathscr{P}_{\mathcal{I}^\bullet}).
\end{equation}

\begin{notation}
In what follows, since substitutions are defined over \emph{terms} (and not elements of free algebras), it is relevant to make the distinction between a term in a term algebra $\term(n)$ and its equivalence class in $\free(n)$. Thus, given a term $t$ over a set of \luk\ variables, we shall write $\overline{t}$ to mean the equivalence class of $t$ in the appropriate free algebra, whenever there is no danger of confusion. Moreover,  to simplify the notation, in an expression such as $[\,\overline{t}\,]_{\mathscr{C}_{\mathcal{E}}\cap\mathscr{P}_{\mathcal{I}^\bullet}}$,  we will substitute the subscript $\mathscr{C}_{\mathcal{E}}\cap\mathscr{P}_{\mathcal{I}^\bullet}$ with simply $\mathcal{I}$. Hence, we write $[\,\overline{t}\,]_\mathcal{I}$ for $[\,\overline{t}\,]_{\mathscr{C}_{\mathcal{E}}\cap\mathscr{P}_{\mathcal{I}^\bullet}}$.
\end{notation}

Given any probabilistic unifier $\sigma$ for $\cc I$, consider ${\sf A}(\sigma) = h_\sigma$ defined as 
\begin{equation}\label{eq:defhsigma}
h_\sigma ([\,\overline{t}\,]_\mathcal{I})= [\,\overline{\sigma^\bullet(t)}\,]_{\mathscr{C}_{\sigma(\mathcal{E})}}
\end{equation}
mapping ${\sf A}(\mathcal{I})$ to the coherent MV-algebra $\free(\sigma(\mathcal{E}))/\mathscr{C}_{\sigma(\mathcal{E})}$.

 \begin{lemma}\label{lemma:alguniff}
 Given a probabilistic unification problem $\cc I$ with a unifier $\sigma$, ${\sf A}(\sigma) = h_\sigma$ is an algebraic unifier for $\sf A(\cc I)$.
 \end{lemma}
 \begin{proof}
 Let $\mathcal{I}$ be as in (\ref{eq:uniProb}) and let  us start considering the maps $\sigma^\bullet : \free(\mathcal{E}) \to \free(\sigma(\mathcal{E}))$ and the natural epimorphism $\mu: \free(\mathcal{E})\to \free(\mathcal{E})/\mathscr{C}_{\mathcal{E}}$. We can then define the homomorphism $\sigma_{\mathcal E}: \free(\mathcal{E})/\mathscr{C}_{\mathcal{E}} \to  \free(\sigma(\mathcal{E}))$ as:
 \begin{equation}\label{eq:sigmae}
 \sigma_{\mathcal E}[\,\overline{t}\,]_{\mathscr{C}_{\mathcal{E}}} =\overline{\sigma^\bullet(t)}.
 \end{equation}
 The map is well-defined because, if $[\,\overline{t}\,]_{\mathscr{C}_{\mathcal{E}}} = [\,\overline{t'}\,]_{\mathscr{C}_{\mathcal{E}}}$,  since $\sigma$ is in particular a probabilistic substitution, $\overline{\sigma^\bullet (t)} = \overline{\sigma^\bullet(t')}$. Thus,  $\sigma_{\mathcal{E}}$ is an homomorphism by the Second Homomorphism Theorem (see \cite[Theorem 6.15]{BS}). 
By Proposition \ref{propFOIntersection}, 
 $$
 {\sf A}(\mathcal{I}) = \free(\mathcal{E})/(\mathscr{C}_{\mathcal{E}}\cap\mathscr{P}_{\mathcal{I}^\bullet}) \cong (\free(\mathcal{E})/\mathscr{C}_{\mathcal{E}})/ \hat{\mathcal{I}},
 $$ 
 where $\hat{\mathcal{I}}$ is the congruence on $\free(\mathcal{E})/\mathscr{C}_{\mathcal{E}}$ generated by the pairs $([\,\overline{t_i}\,]_{\mathscr{C}_{\mathcal{E}}}, [\,\overline{u_i}\,]_{\mathscr{C}_{\mathcal{E}}})$ where $t_i, u_i$ are the terms defining the unification problem as in (\ref{eqipallino}), for $i = 1, \ldots, m$.

 Let us then consider the two natural epimorphisms $\mu_1: (\free(\mathcal{E})/\mathscr{C}_{\mathcal{E}})\to (\free(\mathcal{E})/\mathscr{C}_{\mathcal{E}})/ \hat{\mathcal{I}}$ and $\mu_2:\free(\sigma(\mathcal{E}))\to\free(\sigma(\mathcal{E}))/\mathscr{C}_{\sigma(\mathcal{E})}$ as in the following diagram.

 $$
\xymatrix{
\free(\mathcal{E})/\mathscr{C}_{\mathcal{E}} \ar[r]^{\sigma_\mathcal{E}}\ar[d]_{\mu_1}& \free(\sigma(\mathcal{E}))\ar[r]^{\mu_2}&\free(\sigma(\mathcal{E}))/\mathscr{C}_{\sigma(\mathcal{E})}\\
 (\free(\mathcal{E})/\mathscr{C}_{\mathcal{E}})/ \hat{\mathcal{I}}\ar@{-->}[urr]_{h}& &
}
$$
We now show that $\ker\mu_1\subseteq\ker(\mu_2\circ\sigma_\mathcal{E})$ so that (again by the Second Homomorphism Theorem \cite[Theorem 6.15]{BS}) there exists a homomorphism $h$ closing the diagram. It suffices to show that the generators of $\hat{\mathcal{I}}$ are in $\ker(\mu_2\circ\sigma_\mathcal{E})$. Take then $([\,\overline{t_i}\,]_{\mathscr{C}_{\mathcal{E}}}, [\,\overline{u_i}\,]_{\mathscr{C}_{\mathcal{E}}})$ with $t_i, u_i$ from (\ref{eqipallino}), for any $i = 1 ,\ldots ,m$. It follows directly from (\ref{eq:uniftranslated}) that $[\,\overline{\sigma^\bullet(t_i)}\,]_{\mathscr{C}_{\sigma(\mathcal{E})}} = [\,\overline{\sigma^\bullet(u_i)}\,]_{\mathscr{C}_{\sigma(\mathcal{E})}}$. By the definition of $\sigma_{\mathcal E}$ in (\ref{eq:sigmae}), this yields $\mu_2\circ\sigma_\mathcal{E} ([\,\overline{t_i}\,]_{\mathscr{C}_{\mathcal{E}}})= \mu_2\circ\sigma_\mathcal{E}([\,\overline{u_i}\,]_{\mathscr{C}_{\mathcal{E}}})$. Therefore the generators of $\hat{\mathcal{I}}$ are in $\ker(\mu_2\circ\sigma_\mathcal{E})$ and we can close the diagram.   

Finally, let us call $\iota$ the isomorphism given by Proposition \ref{propFOIntersection}:
$$
\iota: {\sf A}(\mathcal{I}) = \free(\mathcal{E})/(\mathscr{C}_{\mathcal{E}}\cap \mathscr{P}_{\mathcal{I}^\bullet}) \to (\free(\mathcal{E})/\mathscr{C}_{\mathcal{E}})/ \hat{\mathcal{I}}
$$ 
defined as $\iota([\,\overline{t}\,]_{\mathcal I}) = [[\,\overline{t}\,]_{\mathscr{C}_{\mathcal{E}}}]_{\hat{\mathcal{I}}}$. Since $h_\sigma$ is exactly $h \circ \iota$, as it directly follows from the definition of $h_\sigma$ in (\ref{eq:defhsigma}), it is an homomorphism to a coherent MV-algebra and therefore a unifier for ${\sf A}(\mathcal{I})$.
 \end{proof}

In light of the result above, let us define, for every probabilistic unification problem $\mathcal{I}$, $\mathsf{A}(\mathscr{S}_\mathcal{I})$ to be the poset whose universe is $\{h_\sigma\mid \sigma\in \mathscr{S}_\mathcal{I}\}$ and the order is the generality order as in Definition \ref{defAMP}.

Now, consider an algebraic probabilistic unification problem ${\bf F}(\cc E, \cc P) = \free(\cc E)/(\scr{C}_{\cc E} \cap \scr{P}_{\cc P}) $, where
$
\mathcal{P}=\{(t_i[p_{\varphi_{1}},\ldots, p_{\varphi_{k}}], u_i[p_{\varphi_{1}},\ldots, p_{\varphi_{k}}])\mid i=1,\ldots, m\}.
$
We define in the obvious way 
\begin{equation}\label{eqFromAlgToSynt}
{\sf S}( {\bf F}(\cc E, \cc P))=\{t_i[P(\varphi_1),\ldots, P(\varphi_k)] = u_i[P(\varphi_1),\ldots, P(\varphi_k)]\mid i=1,\ldots, m\}. 
\end{equation}
Clearly, ${\sf S}( {\bf F}(\cc E, \cc P))$ is a probabilistic unification problem. We shall now show how to interpret every algebraic unifier for  ${\bf F}(\cc E, \cc P)$ as a probabilistic unifier for ${\sf S}({\bf F}(\cc E, \cc P)$. To do so, let $h$ be a homomorphism of ${\bf F}(\cc E, \cc P)$ to a coherent MV-algebra $\mathbf{C}$. By definition, there exists  a set of events $\mathcal{T}=\{\tau_1,\ldots, \tau_l\}$ such that ${\bf C}$ is isomorphic to $\free(\mathcal{T})/\mathscr{C}_\mathcal{T}$ via a map $\lambda$. Since $\free(\mathcal{T})/\mathscr{C}_\mathcal{T}$ is projective by Corollary \ref{corProjProp}, given the natural epimorphism $j: \free(\cc T) \to \free(\mathcal{T})/\mathscr{C}_\mathcal{T}$  there is an  embedding $i: \free(\mathcal{T})/\mathscr{C}_\mathcal{T} \to \free(\cc T)$  such that $j \circ i = id_{\free(\mathcal{T})/\mathscr{C}_\mathcal{T}}$. We call $s_h$ the homomorphism from $\free(\mathcal{E})$ to $\free(\cc T)$ that is the composition $s_h = i \circ \lambda \circ h \circ \mu$
as clarifed in the following diagram:
 $$
\xymatrix@1{
\free(\cc E) \ar[r]^-{\mu} \ar@/^3.0pc/[rrrr]^{s_{h}}& \free(\cc E) / (\scr C_{\cc E} \cap \scr P_{\cc P}) \ar[r]^-h &\alg C \ar[r]^-\lambda &\free(\cc T) /\scr{C}_{\cc T} \ar@/^/[r]^{i} &\free(\cc T)\ar@/^/[l]_{j}\\
}
$$
Thus, for each $p_{\varphi_{1}}, \ldots, p_{\varphi_{k}}$, consider a term $r_i$ such that $s_h[p_{\varphi_{i}}] = [r_i(p_{\tau_{1}}, \ldots, p_{\tau_{l}})]$, then we define ${\sf S}(h) = \sigma_h$ as:
\begin{equation}\label{eq:defsigmah}
\sigma_h(P(\varphi_{i})) = r_i[P(\tau_1), \ldots, P(\tau_l)]
\end{equation}

\begin{lemma}\label{lemma:symbunif}
Given any algebraic probabilistic unification problem ${\bf F}(\cc E, \cc P)$ with unifier $h$, ${\sf S}(h) = \sigma_{h}$ is a probabilistic unifier for ${\sf S}( {\bf F}(\cc E, \cc P))$.
\end{lemma}
\begin{proof}
We first show that $\sigma_{h}$ as defined in (\ref{eq:defsigmah}) is a probabilistic substitution, that is to say, if 
$$
\vdash_{FP} t[P(\varphi_{1}), \ldots, P(\varphi_{k})] \leftrightarrow u[P(\varphi_{1}), \ldots, P(\varphi_{k})]
$$ 
then 
$$
\vdash_{FP} t[\sigma_{h}(P(\varphi_{1})), \ldots, \sigma_{h}(P(\varphi_{k}))] \leftrightarrow u[\sigma_{h}(P(\varphi_{1})), \ldots, \sigma_{h}(P(\varphi_{k}))].
$$
Suppose  that $\vdash_{FP} t[P(\varphi_{1}), \ldots, P(\varphi_{k})] \leftrightarrow u[P(\varphi_{1}), \ldots, P(\varphi_{k})]$. By Theorem \ref{thm:red1}, this happens if and only if $[\,\overline{t(p_{\varphi_{1}}, \ldots, p_{\varphi_{k}})}\,]_{\scr C_{\cc E}} = [\,\overline{u(p_{\varphi_{1}}, \ldots, p_{\varphi_{k}})}\,]_{\scr C_{\cc E}}$, thus 
$$
[\,\overline{t(p_{\varphi_{1}}, \ldots, p_{\varphi_{k}})}\,]_{\scr C_{\cc E} \cap \scr P_{\cc P}} = [\,\overline{u(p_{\varphi_{1}}, \ldots, p_{\varphi_{k}})}\,]_{\scr C_{\cc E} \cap \scr P_{\cc P}}
$$ 
which implies that $$s_{h}(\,\overline{t(p_{\varphi_{1}}, \ldots, p_{\varphi_{k}})}\,) = s_{h}(\,\overline{u(p_{\varphi_{1}}, \ldots, p_{\varphi_{k}})}\,).$$
Being $s_{h}$ a homomorphism, $t(s_{h}(\,\overline{p_{\varphi_{1}}}\,), \ldots, s_{h}(\,\overline{p_{\varphi_{k}}}\,)) = u(s_{h}(\,\overline{p_{\varphi_{1}}}\,), \ldots, s_{h}(\,\overline{p_{\varphi_{k}}}\,)),$ thus also 
$$
[t(s_{h}(\,\overline{p_{\varphi_{1}}}\,), \ldots, s_{h}(\,\overline{p_{\varphi_{k}}}\,))]_{\scr C_{\tau}} = [u(s_{h}(\,\overline{p_{\varphi_{1}}}\,), \ldots, s_{h}(\,\overline{p_{\varphi_{k}}}\,))]_{\scr C_{\tau}}$$
which is equivalent via Theorem \ref{thm:red1} to what we needed to show, that is  
$$
\vdash_{FP} t[\sigma_{h}(P(\varphi_{1})), \ldots, \sigma_{h}(P(\varphi_{k}))] \leftrightarrow u[\sigma_{h}(P(\varphi_{1})), \ldots, \sigma_{h}(P(\varphi_{k}))].
$$
Notice that the choice of the term $r$ in (\ref{eq:defsigmah}) does not matter because of Proposition \ref{propSubstitution2}.

We proved that $\sigma_{h}$ is a probabilistic substitution, we now prove that it is a unifier for ${\sf S}( {\bf F}(\cc E, \cc P)) = \{t_i[P(\varphi_1),\ldots, P(\varphi_k)] = u_i[P(\varphi_1),\ldots, P(\varphi_k)]\mid i=1,\ldots, m\}$. We need to show that for $i = 1 \ldots m$, 
$$
\vdash_{FP} t_{i}[\sigma_{h}(P(\varphi_{1})), \ldots, \sigma_{h}(P(\varphi_{k}))] \leftrightarrow u_{i}[\sigma_{h}(P(\varphi_{1})), \ldots, \sigma_{h}(P(\varphi_{k}))].
$$
This happens iff  
$$
[t_{i}(s_{h}(\,\overline{p_{\varphi_{1}}}\,), \ldots, s_{h}(\,\overline{p_{\varphi_{k}}}\,))]_{\scr C_{\tau}} = [u_{i}(s_{h}(\,\overline{p_{\varphi_{1}}}\,), \ldots, s_{h}(\,\overline{p_{\varphi_{k}}}\,))]_{\scr C_{\tau}}.
$$
In other words, iff $j \circ s_{h}(\,\overline{t_{i}(p_{\varphi_{1}}, \ldots, \varphi_{k})}\,) =j\circ s_{h} (\,\overline{u_{i}(p_{\varphi_{1}}, \ldots, \varphi_{k})}\,)$. This holds since $$j \circ s_{h} = j \circ i \circ \lambda \circ h \circ \mu = \lambda \circ h \circ \mu$$ and $\mu(\,\overline{t_{i}(p_{\varphi_{1}}, \ldots, \varphi_{k})}\,)= \mu(\,\overline{u_{i}(p_{\varphi_{1}}, \ldots, \varphi_{k})}\,)$ because $t_{i}(p_{\varphi_{1}}, \ldots, \varphi_{k}) = u_{i}(p_{\varphi_{1}}, \ldots, \varphi_{k}) \in \cc P$, thus the proof is completed. 
\end{proof}
We will now show that, given a probabilistic unification problem $$\mathcal{I}=\{t_i[P(\varphi_1),\ldots, P(\varphi_k)] = u_i[P(\varphi_1),\ldots, P(\varphi_k)]\mid i=1,\ldots, m\},$$ its poset of probabilistic unifiers $\scr S_{\cc I}$ is isomorphic to the poset of algebraic unifiers of the algebraic unification problem ${\sf A}(\cc I) = {\bf F}(\cc E, \cc I^{\bullet}) = \free(\mathcal{E})/(\mathscr{C}_{\mathcal{E}}\cap\mathscr{P}_{\mathcal{I}^\bullet}).$
In order to do so, we will prove that the mapping $\sf A$ is surjective on the algebraic unifiers of  ${\sf A}(\cc I)$ and it preserves the partial order. 
We first need the following technical lemma.
\begin{lemma}\label{lemma:probcomp}
Given a probabilistic unification problem $\cc I$, ${\sf S}({\sf A}(\cc I)) = \cc I$, and given any probabilistic unifier $\sigma$, $\sigma \sim \sigma_{h_{\sigma}} = {\sf S}({\sf A}(\sigma))$ in the preorder of unifiers and hence they coincide in the poset  $U_{\cc I}$. Similarly, given any algebraic unification problem $\alg F(\cc E, \cc P)$,  ${\sf A}({\sf S}(\alg F(\cc E, \cc P))) = \alg F(\cc E, \cc P)$ and given any algebraic unifier $k$, $k \sim h_{\sigma_{k}} = {\sf A}({\sf S}(k))$ in the preorder of unifiers and hence they coincide in the poset $U_{\alg F(\cc E, \cc P)}$.
\end{lemma}
\begin{proof}
Let $\cc I$ be as in (\ref{eq:uniProb}), 
$\mathcal{I}=\{t_i[P(\varphi_1),\ldots, P(\varphi_k)] = u_i[P(\varphi_1),\ldots, P(\varphi_k)]\mid i=1,\ldots, m\}.$ 
The facts that ${\sf S}({\sf A}(\cc I)) = \cc I$ and ${\sf A}({\sf S}(\alg F(\cc E, \cc P))) = \alg F(\cc E, \cc P)$ follow directly from the definitions of ${\sf A}$ and ${\sf S}$.

Let $\sigma$ be a probabilistic unifier for $\cc I$, then given the fact that ${\sf S}({\sf A}(\cc I)) = \cc I$ and Lemmas \ref{lemma:alguniff} and \ref{lemma:symbunif}, $\sigma_{h_{\sigma}} $ is also a probabilistic unifier for $\cc I$. We show that $\sigma \sqsubseteq \sigma_{h_{\sigma}}$ and $\sigma_{h_{\sigma}} \sqsubseteq \sigma$. This means that we need to find $\delta, \delta'$ probabilistic substitutions such that $\FPL$ proves that $\sigma = \delta \circ \sigma_{h_{\sigma}}$ and $\sigma_{h_{\sigma}} = \delta' \circ \sigma$ in the sense of (\ref{eq:composition}). It suffices to take $\delta$ and $\delta'$ to be the identity maps on the appropriate set of probabilistic formulas since
$$
\vdash_{FP} t[\sigma(P(\varphi_{1})), \ldots, \sigma(P(\varphi_{k}))] \leftrightarrow t[(\sigma_{h_{\sigma}})(P(\varphi_{1})), \ldots,(\sigma_{h_{\sigma}})(P(\varphi_{k}))].
$$
Indeed, by Theorem \ref{thm:red1} and the definition of $\sigma_{h_{\sigma}}$ (see in particular (\ref{eq:defsigmah})), this is equivalent to saying that for all $\overline{t} \in \free(\cc E)$, $[\,\overline{\sigma^{\bullet}(t)}\,]_{\scr C_{\sigma(\cc E)}} = [s_{h_{\sigma}}(\overline{t})]_{\scr C_{\sigma(\cc E)}}$.

In order to check that this holds, let us unpack the definition of $s_{h_{\sigma}}$. First, recall that $h_\sigma([\,\overline{t}\,]_{\cc I}) =[\,\overline{\sigma^\bullet (t)\,}]_{\scr C_{\sigma(\cc E)}}$. Then, referring to the notation yielding (\ref{eq:defsigmah}), since $h_{\sigma}$ has as codomain a coherent algebra, $\free(\sigma(\mathcal{E}))/\mathscr{C}_{\sigma(\mathcal{E})}$, we can take $\lambda$ to be the identity map. Moreover, we have $i: \free(\sigma(\mathcal{E}))/\mathscr{C}_{\sigma(\mathcal{E})} \to \free(\sigma(\mathcal{E})), j : \free(\sigma(\mathcal{E}))/\mathscr{C} \to \free(\sigma(\mathcal{E}))/\mathscr{C}_{\sigma(\mathcal{E})}$ such that $j \circ i = id$.
Thus, $s_{h_{\sigma}} = i \circ h_\sigma \circ \mu$. Hence:
$$
s_{h_{\sigma}}(\,\overline{t}\,) = i \circ h_{\sigma}\circ \mu(\,\overline{t}\,) = i \circ h_{\sigma}([\,\overline{t}\,]_{\cc I}) = i  ([\,\overline{\sigma^{\bullet}(t)}\,]_{\scr C_{\sigma(\cc E)}}) = i \circ j  (\,\overline{\sigma^{\bullet}(t)}\,).
$$
Therefore we can conclude:
$$[s_{h_{\sigma}}(\,\overline{t}\,)]_{\scr C_{\sigma(\cc E)}} = j \circ i \circ j (\,\overline{\sigma^\bullet(t)}\,) = [\,\overline{\sigma^{\bullet}(t)}\,]_{\scr C_{\sigma(\cc E)}}.$$

We now show that, given $h$ algebraic probabilistic unifier for $\alg A = \alg F(\cc E, \cc P)$, $h: \alg F(\cc E, \cc P) \to \alg C$, $h \sim h_{\sigma_{h}} = {\sf A}({\sf S}(k))$ in the poset of unifiers $U_{\alg F(\cc E, \cc P)}$. That is to say, there are homomorphisms $k, k'$ such that $h = k \circ h_{\sigma_{h}}$ and $h_{\sigma_{h}} = k' \circ h$. Let $\cc I = \vv S(\alg A)$.  
Notice that, following the definitions, if $C \cong \free(\cc T)/ \scr C_{\cc T}$ via a map $\lambda$, then $h_{\sigma_{h}}: \alg F(\cc E, \cc P) \to  \free(\cc T)/ \scr C_{\cc T}$,  and specifically $h_{\sigma_{h}}([\,\overline{t}\,]_{\cc I}) = [s_h(\,\overline{t}\,)]_{\scr C_{\cc T}}$. 
Let us denote again as in the diagram before (\ref{eq:defsigmah}), $s_h = i \circ \lambda \circ h \circ \mu$, and $j$ such that $j \circ i = id$.

We show first that $h_{\sigma_{h}} = \lambda \circ h$. Indeed, for any  $[\,\overline{t}\,]_{\cc I} \in \alg F(\cc E, \cc P)$: 
$$
\lambda \circ h ([\,\overline{t}\,]_{\cc I}) = j \circ i \circ \lambda \circ h ([\,\overline{t}\,]_{\cc I}) = j \circ s_{h} (\,\overline{t}\,)= h_{\sigma_{h}}([\,\overline{t}\,]_{\cc I}).
$$
Thus, since $\lambda$ is an isomorphism, from $h_{\sigma_{h}} = \lambda \circ h$ it also follows that $h = \lambda^{-1} \circ h_{\sigma_{h}}$ and then the proof is completed.

\end{proof}
Given a probabilistic unification problem $\cc I$, we now see $\sf A$ as a map from $(U_{\cc I}, \leq)$
to $(U_{F(\cc E, \cc I^{\bullet}) }, \preceq)$.
\begin{lemma}\label{lemma:posetunif}
Given a probabilistic unification problem $\cc I$, $\sf A$ is surjective on $(U_{F(\cc E, \cc I^{\bullet}) }, \preceq)$ and it preserves the order: given $\sigma$ and $\rho$ probabilistic unifiers for $\cc I$, $\sigma \sqsubseteq \rho$ iff $h_{\sigma} \preceq h_{\rho}$.
\end{lemma}
\begin{proof}
The fact that  $\sf A$ is surjective on $(U_{F(\cc E, \cc I^{\bullet}) }, \preceq)$ follows from Lemma \ref{lemma:probcomp}, indeed given an algebraic unifier $h$ unifier for $\alg F(\cc E, \cc I^{\bullet})$, considering $\sigma_{h}$ unifier for ${\sf S}(\alg F(\cc E, \cc I^\bullet))$, we get that ${\sf A}(\sigma_{h}) = h_{\sigma_{h}}$ coincides with $h$ in the poset $(U_{F(\cc E, \cc I^{\bullet}) }, \preceq)$.

Suppose now that $\sigma \sqsubseteq \rho$, with $Var(\{\sigma(P(\varphi_1)),\ldots, \sigma(P(\varphi_k))\})=\{\tau_{1}, \ldots \tau_{l}\}=\mathcal{T}$ 
and $Var(\{\rho(P(\varphi_1)),\ldots, \rho(P(\varphi_k))\})=\{\gamma_{1}, \ldots \gamma_{m}\}=\mathcal{G}$. 
Then there exists a probabilistic substitution $\delta$ with $Var(\{\delta(P(\gamma_1)),\ldots, \delta(P(\gamma_m))\})=\mathcal{T}$ 
such that $\sigma = \delta \circ \rho$ in the sense of (\ref{eq:composition}). Let us consider the coherent MV-algebra $\free(\rho(\cc E))/ \scr C_{\rho(\cc E)}$. We can see this as the algebraic unification problem $\alg F(\rho(\cc E), \cc P_{\top})$, where $\cc P_{\top}= \{p_{\gamma_{i}} = p_{\gamma_{i}} \mid i = 1, \ldots, m\}$. Indeed, since  $\scr P_{\cc P_{\top}} = [0,1]^{m}$, one has that $$\alg F(\rho(\cc E), \cc P_{\top}) = \free(\rho(\cc E))/ (\scr C_{\rho(\cc E)} \cap \scr P_{\cc P_{\top}})= \free(\rho(\cc E))/ \scr C_{\rho(\cc E)}.$$
Thus, let us consider ${\sf S}(\alg F(\rho(\cc E), \cc P_{\top}))$, that is to say, $\{P(\gamma_{i}) = P(\gamma_{i}) : i = 1 \ldots m \}$.
Therefore, $\delta$ is a probabilistic unifier for ${\sf S}(\alg F(\rho(\cc E), \cc P_{\top}))$, since it is a probabilistic substitution and (\ref{eq:defprobunifier}) is clearly satisfied. We can then define ${\sf A}(\delta) = h_{\delta}: \free(\rho(\cc E))/ \scr C_{\rho(\cc E)} \to \free(\sigma(\cc E))/ \scr C_{\sigma(\cc E)}$, and show that $h_{\sigma} = h_{\delta} \circ h_{\rho}$, which will imply $h_{\sigma} \preceq h_{\rho}$. This holds since for all terms $[\,\overline{t}\,]_{\cc I^{\bullet}} \in {\sf A}(\cc I)$: 
$$
h_{\delta} \circ h_{\rho}([\,\overline{t}\,]_{\cc I^{\bullet}}) = h_{\delta}[\,\overline{\rho^{\bullet}(t)}\,]_{\scr C_{\cc G}} = [\,\overline{\delta^{\bullet} \circ \rho^{\bullet} (t)}\,]_{\scr C_{\cc T}} = [\,\overline{\sigma^{\bullet}(t)}\,]_{\scr C_{\cc T}} = h_{\sigma}([\, \overline{t}\,]_{\cc I^{\bullet}})
$$ 
via the fact that $\sigma = \delta \circ \rho$ in the sense of (\ref{eq:composition}), and Theorem \ref{thm:red1}.

It is now left to prove that if $h_{\sigma} \preceq h_{\rho}$, then $\sigma \sqsubseteq \rho$. Suppose then $h_{\sigma} \preceq h_{\rho}$, i.e. there is a homomorphism $k: \free(\rho(\cc E))/ \scr C_{\rho(\cc E)} \to \free(\sigma(\cc E))/ \scr C_{\sigma(\cc E)}$ such that $h_{\sigma} = k \circ h_{\rho}$. Via the same comment as above, the coherent MV-algebra $\free(\rho(\cc E))/ \scr C_{\rho(\cc E)}$ is the algebraic unification problem $\alg F(\rho(\cc E), \cc P_{\top})$, thus we can consider the probabilistic unifier $\sigma_{k}$ for ${\sf S}(\alg F(\rho(\cc E), \cc P_{\top}))$. We prove that $\sigma_{k} \circ \sigma_{h_{\rho}} = \sigma_{h_{\sigma}}$, which means that $\sigma_{h_{\sigma}} \sqsubseteq \sigma_{h_{\rho}}$, that via Lemma \ref{lemma:probcomp} implies $\sigma \sqsubseteq \rho$. Showing that $\sigma_{k} \circ \sigma_{h_{\rho}} = \sigma_{h_{\sigma}}$ means showing that 
$$
\vdash_{FP} t[\sigma_{h_{\sigma}}(P(\varphi_{1})), \ldots, \sigma_{h_{\sigma}}(P(\varphi_{k}))] \leftrightarrow t[(\sigma_{k} \circ \sigma_{h_{\rho}})(P(\varphi_{1})), \ldots,(\sigma_{k} \circ \sigma_{h_{\rho}})(P(\varphi_{k}))].
$$
This is equivalent to saying that $[\,\overline{\sigma_{h_{\sigma}}^{\bullet}(t)}\,]_{\scr C_{\sigma(\cc E)}} = [\,\overline{\sigma_{k}^{\bullet} \circ \sigma^{\bullet}_{h_{\rho}}(t)}\,]_{\scr C_{\sigma(\cc E)}}$, that is, 
$$
[s_{h_{\sigma}}(\,\overline{t}\,)]_{\scr C_{\sigma(\cc E)}} = [s_{k} \circ s_{h_{\rho}}(\,\overline{t}\,)]_{\scr C_{\sigma(\cc E)}}.
$$ 
In order to prove the latter identity, as clarified by the following diagram, we call $j_{\sigma}$ and $j_{\rho}$ the natural epimorphisms going, respectively, from $ \free(\sigma(\cc E))$ to $ \free(\sigma(\cc E))/\scr C_{\sigma(\cc E)}$ and from $ \free(\rho(\cc E))$ to $ \free(\sigma(\cc E))/\scr C_{\rho(\cc E)}$. Furthermore, let $i_{\sigma} $ and $i_{\rho}$ be the maps (given by the projectivity of the algebras) such that $j_{\sigma} \circ i_{\sigma} = id_{ \free(\sigma(\cc E))/\scr C_{\sigma(\cc E)}}$ and $j_{\rho} \circ i_{\rho} = id_{ \free(\rho(\cc E))/\scr C_{\rho(\cc E)}}.$
$$
\xymatrix{
\free(\cc E) \ar[r]^{\mu} \ar[rdd]_{s_{h_{\rho}}} \ar@/^3.0pc/[rrr]^{s_{h_{\sigma}}}& {\bf F}(\cc E, \cc I^{\bullet}) \ar[r]^-{h_{\sigma}} \ar[d]_{h_{\rho}}& \free(\sigma(\cc E))/\scr C_{\sigma(\cc E)} \ar@/^/[r]^-{i_{\sigma}} & \free(\sigma(\cc E)) \ar@/^/[l]^-{j_{\sigma}}\\
& \free(\rho(\cc E))/\scr C_{\rho(\cc E)} \ar[ur]_{k} \ar@/^/[d]^{i_{\rho}}\\
& \free(\rho(\cc E)) \ar@/^/[u]^{j_{\rho}} \ar[uurr]_{s_{k}}
}
$$
Then, we get that
$$[s_{h_{\sigma}}(\,\overline{t}\,)]_{\scr C_{\sigma(\cc E)}} = j_{\sigma} \circ i_{\sigma} \circ h_{\sigma} \circ \mu (\,\overline{t}\,) = h_{\sigma} \circ \mu (\,\overline{t}\,)$$
and 
$$[s_{k} \circ s_{h_{\rho}}(\,\overline{t}\,)]_{\scr C_{\sigma(\cc E)}}  = j_{\sigma} \circ s_{k} \circ i_{\rho} \circ h_{\rho} \circ \mu (\,\overline{t}\,) = j_{\sigma} \circ i_{\sigma} \circ k \circ j_{\rho} \circ  i_{\rho} \circ h_{\rho} \circ \mu (\,\overline{t}\,) = k \circ h_{\rho} \circ \mu (\,\overline{t}\,).$$
Thus, since by hypothesis $k \circ h_{\rho} = h_{\sigma}$, we have showed that $\sigma_{k} \circ \sigma_{h_{\rho}} = \sigma_{h_{\sigma}}$, that is $\sigma_{h_{\sigma}} \sqsubseteq \sigma_{h_{\rho}}$. Therefore, $\sigma \sqsubseteq \rho$ and the proof is completed. 
\end{proof}
The following result then follows. 
\begin{theorem}\label{thm:probunifghilardi}
Given a (symbolic) probabilistic unification problem $\cc I$ for $\FPL$, there exists an algebraic probabilistic unification problem ${\bf F}(\cc E, \cc P)$ that has a solution or unifier iff $\cc I$ does. Moreover, the respective posets of unifiers are isomorphic.
\end{theorem}
\begin{proof}
By Lemma \ref{lemma:posetunif}, it suffices to consider ${\sf A}(\cc I)$ as the algebraic probabilistic unification problem.
\end{proof}
We recall that by \emph{unification type} for a logic, or for a variety of algebras, we mean the worst unification type occurring in either a symbolic or algebraic problem. Therefore, we can also obtain the following result about probabilistic unification for $\FPL$.
\begin{theorem}\label{thm:unificationtype}
The symbolic and algebraic unification types for $\FPL$ coincide.
\end{theorem}
\begin{proof}
Given a symbolic problem, we find an algebraic problem with the same unification type and vice versa. Indeed, as stated in Theorem \ref{thm:probunifghilardi}, given a probabilistic unification problem $\cc I$ for $\FPL$, there exists an algebraic problem with the same unification type. This is more precisely ${\sf A}(\cc I)$, via Lemma \ref{lemma:posetunif}.

Conversely, let us consider an algebraic probabilistic unification problem, ${\bf F}(\cc E, \cc P)$. Then ${\sf S}( {\bf F}(\cc E, \cc P))$ has the same unification type since, by Lemma \ref{lemma:posetunif}, ${\sf S}( {\bf F}(\cc E, \cc P))$ and ${\sf A}({\sf S}( {\bf F}(\cc E, \cc P)))$ have the same unification type, and moreover ${\sf A}({\sf S}( {\bf F}(\cc E, \cc P))) =  {\bf F}(\cc E, \cc P)$ by Lemma \ref{lemma:probcomp}.
\end{proof}
\subsection{The probabilistic unification type of $\FPL$ is nullary}
In \cite{MS13} the authors adopt Ghilardi's algebraic approach to unification and the geometric description of finitely presented MV-algebras to provide an example showing  the unification problem for \luk\ logic to be of nullary type. Since  $\FPL$ builds on \luk\ logic, if from one side one may expect probabilistic unification problems to be at least as complex as the \luk\ one, from the other, our version of Ghilardi's theorem (namely, Theorems \ref{thm:probunifghilardi} and  \ref{thm:unificationtype}) shows that not all (propositional) unifiers are probabilistic unifiers. Indeed, coherent MV-algebras form a proper subclass of projective ones. Therefore, the worst case scenario depicted in \cite{MS13} does not directly apply here.

Nonetheless, we are going to prove that such pathological example can be adapted to our case and that the probabilistic unification type for the logic $\FPL$ is nullary. 

Let us hence start with a set of two events given by propositional variables $\mathcal{E}=\{x_{1},x_2\}$ and considering the probabilistic unification problem consisting of the single identity:
$$
\mathcal{I}=\{P(x_1)\vee \neg P(x_1)\vee P(x_2)\vee \neg P(x_2)=\top\}
$$ 
Notice that via Lemma \ref{propNumVariables}, this can be equivalently rewritten as a problem where the identities are over the same set of variables, such as 
$\cc I = \{P(x_1)\vee \neg P(x_1)\vee P(x_2)\vee \neg P(x_2) = (P(x_{1}) \to P(x_{1})) \land (P(x_{2}) \to P(x_{2}))\}$.
Moreover, notice that ${\sf A}(\cc I) = {\bf F}(\cc E, \cc I^{\bullet})$ where $\mathcal{I}^\bullet=\{P(x_{1})^\bullet\vee \neg P(x_{1})^\bullet\vee P(x_{2})^\bullet\vee \neg P(x_{2})^\bullet=\top^\bullet\}=\{p_{x_{1}}\vee \neg p_{x_{1}}\vee p_{x_{2}}\vee \neg p_{x_{2}}=\top\}$. This immediately gives the pathological example of \cite{MS13} on propositional variables $\{p_{x_{1}}, p_{x_{2}}\}$. Indeed, in this case the coherence set $\scr C_{\cc E} = [0,1]^{2}$, while $\scr P_{\mathcal{I}^{\bullet}}$ is the border $\mathscr{B}$ of the unit square of $\mathbb{R}^2$ and then 
$$
{\sf A}(\cc I) = \free(\cc E) / ([0,1]^2\cap \scr P_{\mathcal{I}^{\bullet}} )=\free(\cc E) / \scr P_{\mathcal{I}^{\bullet}}.
$$ 
However, the same proof of \cite{MS13} does not directly apply to our case and, in order to exhibit that $\mathcal{I}$ has nullary unification type for $\FPL$, we need to adapt their construction to our case.  Let us first briefly recall the key steps of the construction in \cite{MS13}. The authors define a family of polyhedra $\mathscr{T}_1, \mathscr{T}_2, \ldots$ in $\mathbb{R}^3$ and indexed in $\mathbb{Z}^+$, which is an increasing sequence of squared spirals, each projecting onto the border of the  square $\mathscr{B}$ (see \cite[Fig. 1]{MS13}). Then they  show that each $\mathscr{T}_i$ is $\mathbb{Z}$-homeomorphic to a rational polyhedron $\mathscr{P}_i\subseteq [0,1]^{n_i}$, from some $n_i\in \mathbb{N}$. From the algebraic perspective, each $\mathscr{P}_i$ corresponds to a projective MV-algebra $\free(n_i)/\mathscr{P}_i$, that is the codomain of an algebraic unifier belonging to the $\omega$-chain that gives the pathological example.

The following rephrases the key Lemmas 6.1 and 6.2 from \cite{MS13} in algebraic terms.

\begin{lemma}\label{lemmaRephrasesGlue}
For all $i\in \mathbb{Z}^+$, there exist homomorphisms $h_i: \free(\cc E) / \scr P_{\mathcal{I}^{\bullet}}\to \free(n_i)/\mathscr{P}_i$, and $p_{i+1}:\free(n_{i+1})/\mathscr{P}_{i+1}\to \free(n_i)/\mathscr{P}_i$ such that: 
\begin{enumerate}
\item[(1)] $p_{i+1}\circ h_{i+1}=h_i$;
\item[(2)] For all $i>j$, there is no homomorphism $k: \free(n_j)/\mathscr{P}_j\to \free(n_i)/\mathscr{P}_i$ such that $k\circ h_j= h_i$;
\item[(3)] Let ${\bf P}$ be a projective MV-algebra and let $g: \free(\cc E) / \scr P_{\mathcal{I}^{\bullet}}\to {\bf P}$ be a homomorphism. Then, there exists an index $i_0$ and a homomorphism $g':\free(n_{i_0})/\mathscr{P}_{i_0}\to {\bf P}$ such that $g=g'\circ h_{i_0}$.
\end{enumerate}

\end{lemma}

In order to show that the probabilistic unification problem $\mathcal{I}$ has nullary type, we  proceed as follows. For every $i=1,2,\ldots$, let us denote by $\overline{\mathscr{P}_i}$ the rational polytope generated by $\mathscr{P}_i$ (its convex closure). Each $\overline{\mathscr{P}_i}$ is convex and it contains a Boolean point of $[0,1]^{n_i}$ because, in fact, each $\mathscr{P}_i$ already contains a Boolean point of the same cube by construction. Thus, by Theorem \ref{propZRet} and Definition \ref{defCoheMV}, each $\free(n_i)/\overline{\mathscr{P}_i}$ is a coherent MV-algebra. 

The next lemma, in which we will adopt the notation just introduced, gives us some useful hints on the relation between the projective MV-algebras $\free(n_i)/\mathscr{P}_i$ and $\free(n_i)/\overline{\mathscr{P}_i}$. 
\begin{lemma}\label{lemmaNullary1}
For all $i=1,2,\ldots$, $\mathscr{P}_i$ and $\overline{\mathscr{P}_i}$ are $\mathbb{Z}$-retracts of $[0,1]^{n_i}$ and there are homomorphisms $j_i: \free(n_i)/\mathscr{P}_i\to\free(n_i)/\overline{\mathscr{P}_i}$ and $k_i:\free(n_i)/\overline{\mathscr{P}_i}\to \free(n_i)/\mathscr{P}_i$ such that $j_i$ is injective, $k_i$ is surjective, and $k_i\circ j_i$ is the identity map on $\free(n_i)/\mathscr{P}_i$. 
\end{lemma}
\begin{proof}
Since $\mathscr{P}_i\subseteq\overline{\mathscr{P}_i}$, by the duality in \cite{MS13},  we get that there is a surjective homomorphism $k_i:\free(n_i)/\overline{\mathscr{P}_i}\to \free(n_i)/\mathscr{P}_i$, see also \cite[Theorem 3.5]{Cabrer15}. Since $\free(n_i)/{\mathscr{P}_i}$ is projective (as shown in \cite{MS13}), it follows by the definition of projective algebras that there exists an  homomorphism $j_i: \free(n_i)/\mathscr{P}_i\to\free(n_i)/\overline{\mathscr{P}_i}$ such that $k_i\circ j_i$ is the identity map on $\free(n_i)/\mathscr{P}_i$ and hence $j_i$ is necessarily injective. 
\end{proof}
In the next key lemma the basic  notation  is taken from Figure \ref{figNullary}.
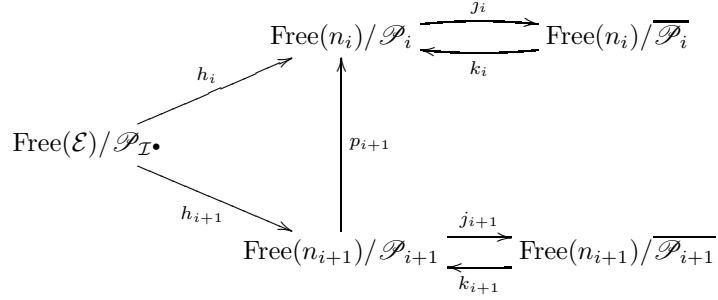
\begin{figure}
$$
\xymatrix{
& \free(n_i)/\mathscr{P}_i\ar@/^/[r]^{j_i}&\free(n_i)/\overline{\mathscr{P}_i}\ar@/^/[l]^{k_i}\\
 \free(\cc E) / \scr P_{\mathcal{I}^{\bullet}}\ar[ur]^{h_i}\ar[dr]_{h_{i+1}}&&\\
& \free(n_{i+1})/\mathscr{P}_{i+1}\ar@/^/[r]^{j_{i+1}}\ar[uu]_{p_{i+1}}&\free(n_{i+1})/\overline{\mathscr{P}_{i+1}}\ar@/^/[l]^{k_{i+1}}\\
}
$$
\caption{Basic construction}
\label{figNullary}
\end{figure}
\begin{lemma}\label{lemmaNullary2}
For all $i=1,2,\ldots$, the following conditions hold.
\vspace{.2cm}

\noindent(1) $j_i\circ h_i=j_i\circ p_{i+1}\circ k_{i+1}\circ j_{i+1}\circ h_{i+1}$;
\vspace{.2cm}

\noindent(2) There is no homomorphism $h: \free(n_{i})/\overline{\mathscr{P}_i}\to \free(n_{i+1})/\overline{\mathscr{P}_{i+1}}$ such that $h\circ j_i\circ h_i=j_{i+1}\circ h_{i+1}$. 
\vspace{.2cm}

\noindent(3) Let ${\bf C}$ be a coherent MV-algebra and let $h: \free(\cc E) / \scr P_{\mathcal{I}^{\bullet}}\to {\bf C}$ be a homomorphism. Then there exists a $i_0$ and a homomorphism $h':  \free(n_{i_0})/\overline{\scr P_{i_0}}\to {\bf C}$ such that $h=h'\circ j_{i_0}\circ h_{i_0}$. 

\end{lemma}
\begin{proof}
(1) Direct inspection shows that both the compositions map $ \free(\cc E) / \scr P_{\mathcal{I}^{\bullet}}$ to $\free(n_i)/\overline{\mathscr{P}_i}$. Moreover, from Lemma \ref{lemmaNullary1}, $k_{i+1}\circ j_{i+1}$ is the identity map on $\free(n_{i+1})/\mathscr{P}_{i+1}$. Thus, the composition on the right-hand side of (1) equals $j_i\circ p_{i+1}\circ  h_{i+1}$ and $p_{i+1}\circ  h_{i+1}=h_i$ by Lemma \ref{lemmaRephrasesGlue} (1). Thus the claim is settled.
\vspace{.2cm}

\noindent(2) Assume by way of contradiction that such $h$ exists and define $h':\free(n_{i})/\mathscr{P}_i\to \free(n_{i+1})/\mathscr{P}_{i+1}$ as $k_{i+1}\circ h\circ j_i$. Then one would have that $h'\circ h_i=k_{i+1}\circ h\circ j_i\circ h_i=k_{i+1}\circ j_{i+1}\circ h_{i+1}$. Again by Lemma \ref{lemmaNullary1}, $k_{i+1}\circ j_{i+1}$ is the identity on $\free(n_{i+1})/\mathscr{P}_{i+1}$ and therefore one would have $h'\circ h_i=h_{i+1}$ contradicting Lemma \ref{lemmaRephrasesGlue} (2). 

$$
\xymatrix{
 \free(n_i)/\mathscr{P}_i\ar@/^/[r]^{j_i}\ar[d]^{h'}&\free(n_i)/\overline{\mathscr{P}_i}\ar@/^/[l]^{k_i}\ar[d]^{h}\\
 \free(n_{i+1})/\scr P_{i+1}\ar@/^/[r]^{j_{i+1}}&\free(n_{i+1})/\overline{\scr P_{i+1}}\ar@/^/[l]^{k_{i+1}}\\
}
$$

\vspace{.2cm}

\noindent(3) Since coherent MV-algebras are projective, by Lemma \ref{lemmaRephrasesGlue} (3), there exists a $i_0$ and a homomorphism $h'': \free(n_{i_0})/\mathscr{P}_{i_0}\to {\bf C}$ such that $h=h''\circ h_{i_0}$. Then, let $h'=h''\circ k_{i_0}$. Thus, $h'\circ j_{i_0}\circ h_{i_0}=h''\circ k_{i_0}\circ j_{i_0}\circ h_{i_0}$. Again by Lemma \ref{lemmaNullary1}, $k_{i_0}\circ j_{i_0}=id$ and hence  $h'\circ j_{i_0}\circ h_{i_0}=h''\circ h_{i_0}=h$. 
$$
\xymatrix{
 \free(\cc E) / \scr P_{\mathcal{I}^{\bullet}} \ar[r]^{h_{i_0}}\ar[ddr]_{h}& \free(n_{i_0})/\scr P_{i_0}\ar@/^/[r]^{j_{i_0}}\ar[dd]_{h''}&\free(n_{i_0})/\overline{\scr P_{i_0}}\ar@/^/[l]^{k_{i_0}}\ar[ddl]_{h'}\\
&&\\
& {\bf C}&\\
}
$$

\end{proof}
\begin{theorem}
The unification type of $\FPL$ is nullary. 
\end{theorem}
\begin{proof}
Consider the probabilistic unification problem $\mathcal{I}$ described at the beginning of this subsection. Its corresponding algebraic unification problem is the finitely presented algebra $\free(\cc E) / \scr P_{\mathcal{I}^{\bullet}}$, and its set of unifiers contains all the homomorphisms $j_i\circ h_i$ (for all $i=1,2,\ldots$). By Lemma \ref{lemmaNullary2}, the set $\{j_i\circ h_i\}_{i\geq 1}$ forms a chain of algebraic unifiers for ${\sf A}(\mathcal{I})$ whose order-type is $\omega$ and which is co-final in the poset $(U_{{\sf A}(\mathcal{I})}, \preceq)$. The claim then follows from Theorem \ref{thm:unificationtype}.
\end{proof}

\section{Conclusions and future work}\label{sec7}
In this paper we presented an encoding of de Finetti's coherence on \luk\ events (as generalized by Mundici in \cite{Mu06}), into propositional \luk\ logic and its equivalent algebraic semantics, the variety of MV-algebras. Via such encoding and a translation map from the modal language of the probability $\FPL$ to propositional \luk\ language, we also proved that deductions of the former can be treated at the propositional level of the latter. Moreover, we isolated a class of projective MV-algebras with respect to which $\FPL$ is complete and, finally, we studied the probabilistic unification problem for  $\FPL$ via algebraic means and proved that it has nullary type.

Our encoding builds on the duality between finitely presented MV-algebras and rational polyhedra developed in \cite{MS12,MS13}, which can be easily shown to specialize to coherent MV-algebras and ($\mathbb{Z}$-homeomorphic images of) coherent sets.  The strong connection between the algebraic and geometric intuitions has been successfully applied in several deep results such as those contained in \cite{CabMu1,MuAdvanced} and it can, in our opinion,  be further and systematically explored to strengthen the link between probability, algebra and logic. 

Future work on this subject may explore several  directions. In particular, from the algebraic perspective, the two-sorted approach developed in the recent paper \cite{KM} surely needs to be further investigated, and its relation with coherent MV-algebras to be better understood. 
Moreover, in a similar direction, it would be interesting to show whether $\FPL$ is algebraizable (in a sense that necessarily extends the classical Blok and Pigozzi definition \cite{BP89}).

Concerning algebraizable probability logics, the formal system called ${\rm S}\FPL$ in \cite{FM09} is an algebraizable extension of $\FPL$, and its equivalent algebraic semantics is given by the variety $\mathsf{SMV}$ of {\em MV-algebras with an internal state}. However, much less is known for ${\rm S}\FPL$. For instance, it is an open problem to show its standard completeness. Moreover, it would be interesting to understand up to which extent the results presented in the present paper for $\FPL$ can be extended to the more general  ${\rm S}\FPL$.

Finally, concerning probabilistic unification, we already pointed out that our approach focuses on the outer language and treats atomic modal formulas of the form $P(\varphi)$ as variables that have to be coherently evaluated. However, the two-tiered nature of the language of $\FPL$ suggests that another {\em internal} probabilistic unification problem could be investigated.  With the latter we mean the following: consider a (symbolic) probabilistic unification problem $\mathcal{I}=\{t_i[P(\varphi_1),\ldots, P(\varphi_k)] = u_i[P(\varphi_1),\ldots, P(\varphi_k)]\mid i=1,\ldots, m\}$ as in Definition \ref{def:SPUI}. Then, by an {\em internal} unifier for $\mathcal{I}$, one can consider a \luk\ substitution $\sigma$ from the propositional variables occurring in the events $\varphi_1,\ldots,\varphi_k$ such that 
$$
\vdash_{FP} t_i[P(\sigma\varphi_1),\ldots, P(\sigma\varphi_k)] \leftrightarrow u_i[P(\sigma\varphi_1),\ldots, P(\sigma\varphi_k)]. 
$$ 
Notice that the above problem does not reduce to the probabilistic unification problem that we consider in Section \ref{sec6}. In fact, since the operator $P$ is not truth functional, there is no way, in general, to reduce an atomic modal formula of the form $P(\sigma\varphi)$ to $\sigma P(\varphi)$. Therefore, we will need to develop alternative techniques to approach it.

All the research directions we mentioned above will surely need a deeper understanding of the relationships between logic, algebra, geometry and uncertainty that the present paper has hopefully contributed to grasp.

\section*{Funding}
The authors acknowledge partial support by the MOSAIC project (H2020-MSCA-RISE-2020 Project 101007627). Ugolini acknowledges support from the Marie Sk\l odowska-Curie grant agreement No 890616 (H2020-MSCA-IF-2019), and the Ramon y Cajal programme RyC2021-032670-I. Flaminio acknowledges support by the Spanish project PID2019-111544GB - C21/AEI/10.13039/501100011033.


\end{document}